\documentclass[11pt,twoside]{amsart}
\pdfoutput=1
\usepackage[english]{babel}
\usepackage{afterpage}
\usepackage{csquotes}
\usepackage{amsfonts}
\usepackage{etoolbox}
\usepackage{amsmath}
\usepackage{mathabx}
\usepackage{amssymb}
\usepackage{amsthm}
\usepackage{stmaryrd}
\usepackage{dsfont}
\usepackage{pifont}
\usepackage{caption}
\usepackage{subcaption}
\usepackage{mathrsfs}
\usepackage{graphicx}
\usepackage[all]{xy}
\usepackage{enumitem}
\usepackage{geometry}
\usepackage{amsfonts}
\usepackage{fancyhdr}
\usepackage{calc}
\usepackage{cancel}
\usepackage{accents}
\usepackage{abraces}
\usepackage{amsaddr}
\usepackage{tikz,tikz-3dplot}
\usetikzlibrary{decorations.markings}
\usetikzlibrary{shapes,arrows,plotmarks}

\usepackage{ulem}

\usepackage{adjustbox}
\usepackage{url}
\usepackage[english]{babel}
\usepackage{csquotes}
\usepackage{amsfonts}
\usepackage{amsmath}
\usepackage{amssymb}
\usepackage{amsthm}
\usepackage{caption}
\usepackage{subcaption}
\usepackage{graphicx}
\usepackage{geometry}
\usepackage{amsfonts}
\usepackage{calrsfs}
\usepackage{accents}
\usepackage{fancyhdr}
\usepackage[new]{old-arrows}

\usepackage{changepage}
\usepackage{pdflscape}
\usepackage{afterpage}
\usepackage{flafter}
\usepackage{fancyhdr}

\usepackage{hyperref}

\fancypagestyle{lscape}{% 
\fancyhf{} % clear all header and footer fields 
\fancyfoot[LE]{}
\fancyfoot[LO] {}
 
}

\definecolor{cadmiumgreen}{rgb}{0.0, 0.42, 0.24}

\usepackage{color}

\tdplotsetmaincoords{70}{165}

\newcommand{\changefont}{%
    \fontsize{8}{8}\selectfont
}
\theoremstyle{plain}
\newtheorem{prop}{Proposition}[section]
\newtheorem{prop-def}[prop]{Proposition-Definition}

\newtheorem{lem}[prop]{Lemma}
\newtheorem{theo}[prop]{Theorem}
\newtheorem{cor}[prop]{Corollary}

\newtheorem*{prop*}{Proposition}
\newtheorem*{prop-def*}{Proposition-Definition}
\newtheorem*{propri*}{Property}
\newtheorem*{lem*}{Lemma}
\newtheorem*{theo*}{Theorem}
\newtheorem*{cor*}{Corollary}

\newcommand{\lra}{\longrightarrow}

\newcommand{\ra}{\rightarrow}
\newcommand{\sdp}{\times\kern-.2em\vrule height1.1ex depth-.05ex}
\newcommand{\epi}{\lra \kern-.8em\ra}

\newcommand{\R}{{\mathbb R}}

\newcommand{\Z}{{\mathbb Z}}

\DeclareMathOperator{\codim}{codim}

\DeclareMathOperator{\Aut}{Aut}

\DeclareMathOperator{\Ker}{Ker}
\DeclareMathOperator{\Ima}{Im}
\DeclareMathOperator{\id}{id}

\newcommand{\E}{\mathrm{E}}
\renewcommand{\O}{\mathrm{O}}

\newcommand{\K}{\mathcal{K}}
\newcommand{\A}{\mathcal{A}}

\newcommand{\Sym}{\mathfrak{S}}

\DeclareMathOperator\conv{conv}
\DeclareMathOperator\vertices{vert}

 \normalbaselines
\makeatletter
\newlength\@SizeOfCirc%
\newcommand{\CircleArrowRight}[1]{%
    \setlength{\@SizeOfCirc}{\maxof{\widthof{#1}}{Veightof{#1}}}%
    \tikz [x=1.0ex,y=1.0ex,line width=.12ex]%
        \draw [->,anchor=center]%
            node (0,0) {#1}%
            (0,0.8\@SizeOfCirc) arc (85:-240:0.8\@SizeOfCirc);%
}%
\newcommand{\CircleArrowLeft}[1]{%
    \setlength{\@SizeOfCirc}{\maxof{\widthof{#1}}{Veightof{#1}}}%
    \tikz [x=1.0ex,y=1.0ex,line width=.12ex]%
        \draw [<-,anchor=center]%
            node (0,0) {#1}%
            (0,0.8\@SizeOfCirc) arc (85:-240:0.8\@SizeOfCirc);%
}%
\makeatother

\newtheorem{theoA}{Theorem}

\theoremstyle{definition}
\newtheorem*{rem*}{Remark}
\newtheorem*{definition*}{Definition}
\newtheorem*{example*}{Example}
\newtheorem*{notation*}{Notation}

\newtheorem{example}[prop]{Example}
\newtheorem{rem}[prop]{Remark}
\newtheorem{definition}[prop]{Definition}

\title[Fundamental polytope for the isometry group of an alcove]
{Fundamental polytope for the isometry group\\ of an alcove}

\author{Lucas Seco}
\address{Universidade de Brasília, Brazil}
\address{\texttt{lseco@unb.br}}
\author{Arthur Garnier}
\address{Universit\'e de Picardie Jules Verne, France}
\address{\texttt{arthur.garnier@math.cnrs.fr}}
\author{Karl-Hermann Neeb}
\address{Friedrich-Alexander-Universit\"at Erlangen-Nürnberg, Germany}
\address{\texttt{neeb@math.fau.de}}

\date{\today}


\begin{document}

\begin{abstract}
A fundamental alcove $\A$ is a tile in a paving of a vector space $V$ by an affine reflection group $W_{\mathrm{aff}}$. Its geometry encodes essential features of $W_{\mathrm{aff}}$, such as its affine Dynkin diagram $\widetilde{D}$ and fundamental group $\Omega$. 
In this article we investigate its full isometry group $\Aut(\A)$. 
It is well known that the isometry group of a regular polyhedron
is generated by hyperplane reflections on its faces. Being a simplex, an alcove $\A$ is the simplest of polyhedra, nevertheless it is seldom a regular one.
In our first main result we show that $\Aut(\A)$ is isomorphic to $\Aut(\widetilde{D})$. Building on this connection, we establish that $\Aut(\A)$
is an abstract Coxeter group, with generators given by affine isometric involutions of the ambient space. 
Although these involutions are seldom reflections,
our second main result leverages them to construct,
by slicing the Komrakov--Premet fundamental polytope $\K$
for the action of $\Omega$, a family of
fundamental polytopes for the action of $\Aut(\A)$ on~$\A$,
whose vertices are contained in the vertices of $\K$
and whose faces are parametrized by the so-called balanced minuscule roots, which we introduce here.  In an appendix, we discuss some related negative results on stratified centralizers and equivariant triangulations.
\end{abstract}

\maketitle


\setcounter{tocdepth}{1}
\tableofcontents

\section{Introduction}

Let $V$ be a finite dimensional euclidean space with inner product $(\cdot, \cdot)$, and corresponding norm $\| \cdot \|$. Denote by $\E(V) := V \rtimes \O(V)$ its Euclidean group of rigid motions, and by $\pi: \E(V) \to \O(V)$ the projection to the linear part. 
For $v \in V$ and $c \in \R$, denote by $\{ v = c \}$
the affine hyperplane $\{x \in V \mid (v,x) = c \}$.
A fundamental domain of a subgroup $G \subseteq \E(V)$ is a closed and connected subset $\mathcal{F} \subseteq V$ such that: for every $x \in V$, we have $Gx \cap \mathcal{F} \neq \emptyset$, and for every $g \in G\setminus\{1\}$,
the intersection $g \mathcal{F} \cap \mathcal{F}$ has empty interior.

A fundamental alcove $\A$ is a tile in a paving of $V$ by reflections and translations of an affine reflection group $W_{\mathrm{aff}}$, which we assume irreducible (see Section \ref{sec:prelim} for setup and notation).
In this article we investigate its full isometry group $\Aut(\A)$. 
It is well known that the isometry group of regular polyhedra is generated by affine
hyperplane reflections. Being a simplex, an alcove $\A$ is the simplest of polyhedra, nevertheless it is seldom a regular one.
In our first main result, Theorem A stated below, we show that $\Aut(\A)$ is isomorphic to the automorphism group of the affine Dynkin diagram of $W_{\mathrm{aff}}$. 
Building on this connection, we establish that $\Aut(\A)$ is an abstract Coxeter group with generators given by affine isometric involutions of the ambient space. Although these involutions are seldom reflections,
in our second main result, Theorem B stated below, we leverage them to obtain a family of fundamental polytopes for the action of $\Aut(\A)$ on $\A$.
We now proceed to state our main results.

The isometry group of a simplex $\sigma$ of $V$ is 
$$
\Aut(\sigma) := \{ \zeta \in \E(V) \mid \zeta(\sigma) = \sigma \}.
$$
Consider the simplex $\nu$ whose vertices are the inward unit normals to the faces of $\sigma$.
Given an isometry $\zeta$ of $\sigma$, its linear part $\pi(\zeta)$ is an isometry of $\nu$. Thus projection to the linear part restricts to the group homomorphism
\begin{equation}
\label{eq:angle-angle} 
\pi: \Aut(\sigma) \to \Aut(\nu)
\end{equation}
which is actually an isomorphism (see Proposition \ref{prop:iso-angle-angle}), where
$$
\Aut(\nu) = \{ \phi \in \O(V) \mid \phi(\nu) = \nu \}.
$$
It can be seen as a generalization of the angle-angle-angle criterion for congruence to a plane triangle and it provides a natural way to ``linearize'' $\Aut(\sigma)$ as a first step for its investigation.

For a fundamental alcove $\A$ of $W_{\mathrm{aff}}$, its isometry group
\begin{equation}
  \Aut(\A) = \{ \zeta \in \E(V) \mid \zeta(\A) = \A \}
\end{equation}
is related to the affine structure of $W_{\mathrm{aff}}$ as follows.
Consider the affine Dynkin diagram $\widetilde{D}$ of the affine reflections group $W_{\mathrm{aff}}$,
the Dynkin diagram $D$ of its Weyl group $W$,
and the so called\begin{footnote}{
At this point the term ``fundamental group'' may appear
misleading, but this group actually can be interpreted as a fundamental
group of a suitable centerfree compact Lie group.}\end{footnote}
fundamental group $\Omega$ of $W_{\mathrm{ext}}$.
%Since $\Aut(D)$ stabilizes $\Pi$ and $\alpha_0$,  it also stabilizes $\A$, so that $\Aut(D) \leq \Aut(\A)$.
%By definition, $\Omega \leq \Aut(\A)$.
Note that the inward normals of $\A$ give rise to a Coxeter diagram $\widetilde{C}$ which comes from the affine Dynkin diagram $\widetilde{D}$ of $W_{\mathrm{ext}}$. 
Indeed, we can project a Dynkin diagram onto a Coxeter diagram by preserving the vertices but forgetting the direction of the edges and replacing the edge weights by the corresponding labels. 

For constrast we also consider the isometry group of the corresponding Weyl chamber $\mathcal{C} = \R^+ \A$, given by
\begin{equation}
  \Aut(\mathcal{C})
  := \{ \zeta \in \E(V) \mid \zeta(\mathcal{C}) = \mathcal{C} \}
= \{ \phi \in \O(V) \mid \phi(\mathcal{C}) = \mathcal{C} \}
\end{equation}
Note that the inward normals of $\mathcal{C}$ give rise to a Coxeter diagram $C$ which comes from the Dynkin diagram $D$ of $W$.

\begin{center}
\def\svgwidth{19cm}
%% Creator: Inkscape 1.0.1 (c497b03c, 2020-09-10), www.inkscape.org
%% Accompanies image file 'B2-chamber-alcove.pdf' (pdf, eps, ps)
\begingroup%
  \makeatletter%
  \providecommand\color[2][]{%
    \errmessage{(Inkscape) Color is used for the text in Inkscape, but the package 'color.sty' is not loaded}%
    \renewcommand\color[2][]{}%
  }%
  \providecommand\transparent[1]{%
    \errmessage{(Inkscape) Transparency is used (non-zero) for the text in Inkscape, but the package 'transparent.sty' is not loaded}%
    \renewcommand\transparent[1]{}%
  }%
  \providecommand\rotatebox[2]{#2}%
  \newcommand*\fsize{\dimexpr\f@size pt\relax}%
  \newcommand*\lineheight[1]{\fontsize{\fsize}{#1\fsize}\selectfont}%
  \ifx\svgwidth\undefined%
    \setlength{\unitlength}{259.51646559bp}%
    \ifx\svgscale\undefined%
      \relax%
    \else%
      \setlength{\unitlength}{\unitlength * \real{\svgscale}}%
    \fi%
  \else%
    \setlength{\unitlength}{\svgwidth}%
  \fi%
  \global\let\svgwidth\undefined%
  \global\let\svgscale\undefined%
  \makeatother%
  \begin{picture}(1,0.21206107)%
    \lineheight{1}%
    \setlength\tabcolsep{0pt}%
    \put(0,0){\includegraphics[width=\unitlength,page=1]{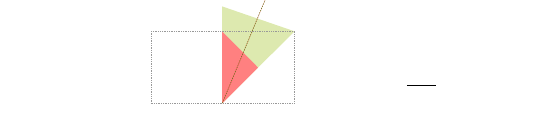}}%
    \put(0.55197217,0.02124167){\makebox(0,0)[lt]{\lineheight{1.25}\smash{\begin{tabular}[t]{l}$\alpha_1$\end{tabular}}}}%
    \put(0.25719319,0.15996124){\makebox(0,0)[lt]{\lineheight{1.25}\smash{\begin{tabular}[t]{l}$\alpha_2$\end{tabular}}}}%
    \put(0.55197217,0.15418126){\color[rgb]{0.6,0.6,0.6}\makebox(0,0)[lt]{\lineheight{1.25}\smash{\begin{tabular}[t]{l}$\alpha_0$\end{tabular}}}}%
    \put(0.42410548,0.09512864){\makebox(0,0)[lt]{\lineheight{1.25}\smash{\begin{tabular}[t]{l}$\mathcal{A}$\end{tabular}}}}%
    \put(0.43566545,0.15292845){\makebox(0,0)[lt]{\lineheight{1.25}\smash{\begin{tabular}[t]{l}$\mathcal{C}$\end{tabular}}}}%
    \put(0,0){\includegraphics[width=\unitlength,page=2]{B2-chamber-alcove.pdf}}%
    \put(0.69647173,0.14262133){\makebox(0,0)[lt]{\lineheight{1.25}\smash{\begin{tabular}[t]{l}$0$\end{tabular}}}}%
    \put(0.74849126,0.14262133){\makebox(0,0)[lt]{\lineheight{1.25}\smash{\begin{tabular}[t]{l}$1$\end{tabular}}}}%
    \put(0.80051122,0.14262133){\makebox(0,0)[lt]{\lineheight{1.25}\smash{\begin{tabular}[t]{l}$2$\end{tabular}}}}%
    \put(0,0){\includegraphics[width=\unitlength,page=3]{B2-chamber-alcove.pdf}}%
    \put(0.74849126,0.08482151){\makebox(0,0)[lt]{\lineheight{1.25}\smash{\begin{tabular}[t]{l}$1$\end{tabular}}}}%
    \put(0.80051122,0.08482151){\makebox(0,0)[lt]{\lineheight{1.25}\smash{\begin{tabular}[t]{l}$2$\end{tabular}}}}%
    \put(0,0){\includegraphics[width=\unitlength,page=4]{B2-chamber-alcove.pdf}}%
    \put(0.84675127,0.15996124){\makebox(0,0)[lt]{\lineheight{1.25}\smash{\begin{tabular}[t]{l}$\widetilde{D}$\end{tabular}}}}%
    \put(0.84675127,0.10216143){\makebox(0,0)[lt]{\lineheight{1.25}\smash{\begin{tabular}[t]{l}$D$\end{tabular}}}}%
    \put(0.84675127,0.04436165){\makebox(0,0)[lt]{\lineheight{1.25}\smash{\begin{tabular}[t]{l}$C$\end{tabular}}}}%
    \put(0,0){\includegraphics[width=\unitlength,page=5]{B2-chamber-alcove.pdf}}%
    \put(0.77739119,0.06170161){\makebox(0,0)[lt]{\lineheight{1.25}\smash{\begin{tabular}[t]{l}$\text{\tiny 4}$\end{tabular}}}}%
    \put(0.74849126,0.02702174){\makebox(0,0)[lt]{\lineheight{1.25}\smash{\begin{tabular}[t]{l}$1$\end{tabular}}}}%
    \put(0.80051122,0.02702174){\makebox(0,0)[lt]{\lineheight{1.25}\smash{\begin{tabular}[t]{l}$2$\end{tabular}}}}%
    \put(0.41828857,0.06013359){\color[rgb]{1,1,1}\makebox(0,0)[lt]{\lineheight{1.25}\smash{\begin{tabular}[t]{l}$\mathcal{K}$\end{tabular}}}}%
  \end{picture}%
\endgroup%

The situation for the root system $B_2$.
\end{center}
\vspace{5mm}

\begin{theoA}
\label{thm:A}
The automorphism groups of ${\mathcal C}, C$, $\A$, $D$ and $\widetilde{D}$
are related as follows:
\begin{enumerate}[label=\roman*)]
\item[\rm(i)] $\Aut(\mathcal{C}) = \Aut(C)$.

\item[\rm(ii)] The projection 
    $\pi: \E(V) \to \O(V)$ to the linear part restricts to
    an isomorphism $\Aut(\A) \to \Aut(\widetilde{D})$ which is
    the identity on $\Aut(D)$.

\item[\rm(iii)] $\Aut(\A) = \Omega \rtimes \Aut(D)$

\item[\rm(iv)] $\Aut(\A)$ is an abstract Coxeter group,
  generated by the affine involutions of $V$ in Table \ref{table:thm1}.

%\item $\Aut(\A)$ normalizes $W_{\mathrm{aff}}$. 
\end{enumerate}
\end{theoA}

\afterpage{
\begin{landscape}
\thispagestyle{lscape}
\pagestyle{lscape}
\begin{figure}
\vspace{-24pt}
\includegraphics[width=20cm]{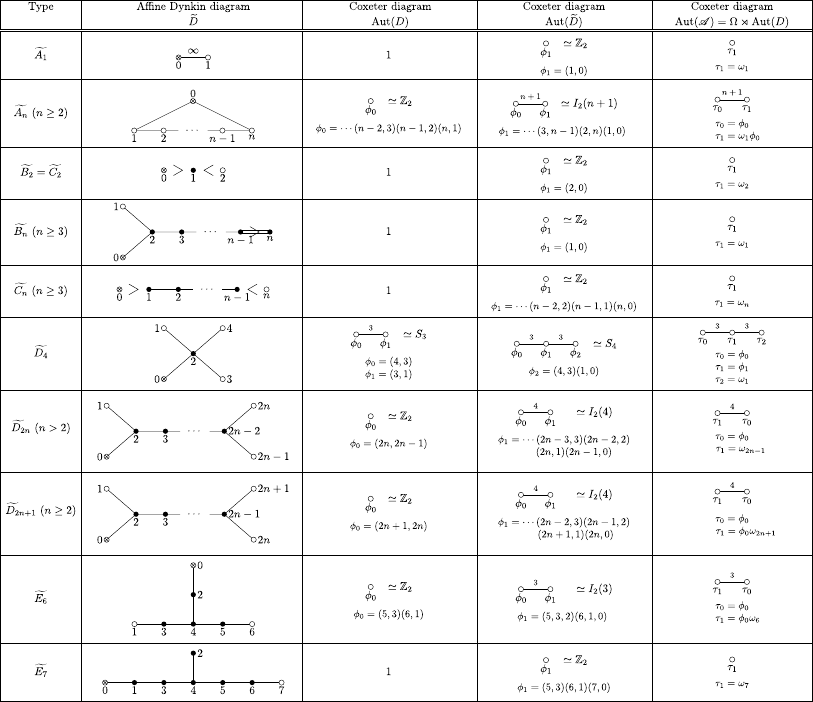}
\captionof{table}{Automorphisms of Dynkin diagrams and their
  Coxeter realizations as affine isometries of the fundamental alcove: permutations are written from right to left. $\widetilde{E_8}$, $\widetilde{F_4}$, $\widetilde{G_2}$ have trivial diagram automorphisms.}
\label{table:thm1}
\end{figure}
\end{landscape}
\clearpage
}

Thus, while the geometry of the Weyl chamber loses data from the Dynkin diagram automorphisms, the geometry of the fundamental alcove preserves this data, and even more: it preserves the data of the affine Dynkin diagram automorphisms.
This is illustrated above for type $B_2$: the (linear) symmetry of the Weyl chamber comes from $\Aut(C)$ and the (affine) symmetry of the alcove comes from $\Aut(\widetilde{D})$, while $\Aut(D) = 1$.

Even though the affine involutions that generate $\Aut(\A)$ in Theorem A are in general not hyperplane reflections, in the second main result of this article we use them
to obtain a fundamental polytope $\mathcal{L}$ for the action of $\Aut(\A)$ on $\A$, as follows. Theorem A gives us $\Aut(\A) = \Omega \rtimes \Aut(D)$. Start with a fundamental domain for the action of $\Omega$ on $\A$ given by the Komrakov--Premet polytope (see \cite{Ga24})
\begin{equation}
\label{eq:kp-domain}
\mathcal{K} :=\{x \in \A \mid (\alpha_0+\alpha_j,x)\leq 1 
~\text{ for all minuscule }\alpha_j
\} 
%\\ %&=\conv\left(\{0\}\cup\left\{\frac{\varpi^\vee_i}{n_i}\right\}_{i\in I\setminus J}\cup\left\{\frac{1}{|J'|+1}\sum_{j\in J'}\varpi^\vee_j\right\}_{J'\subseteq J}\right).
\end{equation}
%where $J = \{ j \in I: \, n_j = 1 \}$. 
Since $\Aut(D)$ stabilizes $\Pi$ and $\alpha_0$, it leaves $\K$ invariant.
Thus, for $\mathcal{L}$ to be a fundamental polytope for the $\Aut(\A)$ action
on $\A$, it is enough that $\mathcal{L}$ is a fundamental polytope for the $\Aut(D)$ action on $\K$.

\begin{center}
\def\svgwidth{19cm}
%% Creator: Inkscape 1.0.1 (c497b03c, 2020-09-10), www.inkscape.org
%% Accompanies image file 'A2-vertex-dynkin.pdf' (pdf, eps, ps)
\begingroup%
  \makeatletter%
  \providecommand\color[2][]{%
    \errmessage{(Inkscape) Color is used for the text in Inkscape, but the package 'color.sty' is not loaded}%
    \renewcommand\color[2][]{}%
  }%
  \providecommand\transparent[1]{%
    \errmessage{(Inkscape) Transparency is used (non-zero) for the text in Inkscape, but the package 'transparent.sty' is not loaded}%
    \renewcommand\transparent[1]{}%
  }%
  \providecommand\rotatebox[2]{#2}%
  \newcommand*\fsize{\dimexpr\f@size pt\relax}%
  \newcommand*\lineheight[1]{\fontsize{\fsize}{#1\fsize}\selectfont}%
  \ifx\svgwidth\undefined%
    \setlength{\unitlength}{166.20510657bp}%
    \ifx\svgscale\undefined%
      \relax%
    \else%
      \setlength{\unitlength}{\unitlength * \real{\svgscale}}%
    \fi%
  \else%
    \setlength{\unitlength}{\svgwidth}%
  \fi%
  \global\let\svgwidth\undefined%
  \global\let\svgscale\undefined%
  \makeatother%
  \begin{picture}(1,0.28224976)%
    \lineheight{1}%
    \setlength\tabcolsep{0pt}%
    \put(0,0){\includegraphics[width=\unitlength,page=1]{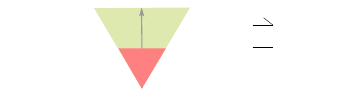}}%
    \put(0.6206375,0.1499143){\makebox(0,0)[lt]{\lineheight{1.25}\smash{\begin{tabular}[t]{l}$\alpha_1$\end{tabular}}}}%
    \put(0.40493749,0.26813944){\color[rgb]{0.6,0.6,0.6}\makebox(0,0)[lt]{\lineheight{1.25}\smash{\begin{tabular}[t]{l}$\alpha_0$\end{tabular}}}}%
    \put(0,0){\includegraphics[width=\unitlength,page=2]{A2-vertex-dynkin.pdf}}%
    \put(0.17662773,0.1499143){\makebox(0,0)[lt]{\lineheight{1.25}\smash{\begin{tabular}[t]{l}$\alpha_2$\end{tabular}}}}%
    \put(0.41436352,0.07309756){\makebox(0,0)[lt]{\lineheight{1.25}\smash{\begin{tabular}[t]{l}$\mathcal{L}$\end{tabular}}}}%
    \put(0,0){\includegraphics[width=\unitlength,page=3]{A2-vertex-dynkin.pdf}}%
    \put(0.40105926,0.11835835){\makebox(0,0)[lt]{\lineheight{1.25}\smash{\begin{tabular}[t]{l}$\mathcal{A}$\end{tabular}}}}%
    \put(0.38654852,0.07309756){\color[rgb]{1,1,1}\makebox(0,0)[lt]{\lineheight{1.25}\smash{\begin{tabular}[t]{l}$\mathcal{K}$\end{tabular}}}}%
    \put(0.30178424,0.23568331){\makebox(0,0)[lt]{\lineheight{1.25}\smash{\begin{tabular}[t]{l}$\mathcal{C}$\end{tabular}}}}%
    \put(0,0){\includegraphics[width=\unitlength,page=4]{A2-vertex-dynkin.pdf}}%
    \put(0.72510184,0.18644){\makebox(0,0)[lt]{\lineheight{1.25}\smash{\begin{tabular}[t]{l}$1$\end{tabular}}}}%
    \put(0.78195953,0.18644){\makebox(0,0)[lt]{\lineheight{1.25}\smash{\begin{tabular}[t]{l}$2$\end{tabular}}}}%
    \put(0,0){\includegraphics[width=\unitlength,page=5]{A2-vertex-dynkin.pdf}}%
    \put(0.72510184,0.12326505){\makebox(0,0)[lt]{\lineheight{1.25}\smash{\begin{tabular}[t]{l}$1$\end{tabular}}}}%
    \put(0.78195953,0.12326505){\makebox(0,0)[lt]{\lineheight{1.25}\smash{\begin{tabular}[t]{l}$2$\end{tabular}}}}%
    \put(0,0){\includegraphics[width=\unitlength,page=6]{A2-vertex-dynkin.pdf}}%
    \put(0.83249962,0.19960536){\makebox(0,0)[lt]{\lineheight{1.25}\smash{\begin{tabular}[t]{l}$\widetilde{D}$\end{tabular}}}}%
    \put(0.83249962,0.13643041){\makebox(0,0)[lt]{\lineheight{1.25}\smash{\begin{tabular}[t]{l}$D$\end{tabular}}}}%
    \put(0.83249962,0.07325553){\makebox(0,0)[lt]{\lineheight{1.25}\smash{\begin{tabular}[t]{l}$C$\end{tabular}}}}%
    \put(0.7587487,0.09220798){\makebox(0,0)[lt]{\lineheight{1.25}\smash{\begin{tabular}[t]{l}$\text{\tiny 3}$\end{tabular}}}}%
    \put(0.72510184,0.06009017){\makebox(0,0)[lt]{\lineheight{1.25}\smash{\begin{tabular}[t]{l}$1$\end{tabular}}}}%
    \put(0.78195953,0.06009017){\makebox(0,0)[lt]{\lineheight{1.25}\smash{\begin{tabular}[t]{l}$2$\end{tabular}}}}%
    \put(0,0){\includegraphics[width=\unitlength,page=7]{A2-vertex-dynkin.pdf}}%
    \put(0.7553078,0.24129091){\makebox(0,0)[lt]{\lineheight{1.25}\smash{\begin{tabular}[t]{l}$0$\end{tabular}}}}%
    \put(0,0){\includegraphics[width=\unitlength,page=8]{A2-vertex-dynkin.pdf}}%
  \end{picture}%
\endgroup%

The situation for the root system $A_2$.
% OK \red{KH: The Coxeter diagram should have $3$ instead of $4$.}
\end{center}
\vspace{5mm}

The idea is to observe that a nontrivial involution $\phi_0$ generating $\Aut(D)$ looks like a planar reflection of the Dynkin diagram and use this to
specify a vector $v_0 \in V$ that plays the role of a root for this involution, as follows. We say that $\Pi_0 \subseteq \Pi$ is a {\it balanced set of roots}
w.r.t.\ $\phi_0$ if it splits as a disjoint union $\Pi_0 = \Pi_+ \cup \Pi_-$ such that $\phi_0(\Pi_+) = \Pi_-$.
Thus, $\Pi_0$ is balanced if and only if it is $\phi_0$-invariant and contains no $\phi_0$-fixed point.
The corresponding balanced minuscule root is the vector
\begin{equation}
v_0 := \sum_{\alpha \in \Pi^+} \alpha - \sum_{\alpha \in \Pi^-} \alpha
 = \sum_{\alpha \in \Pi^+} (\alpha - \phi_0(\alpha)) 
\end{equation}
so that clearly $v_0 \neq 0$ satisfies $\phi_0(v_0) = -v_0$. 
%We say that $v_0$ is minuscule if $\Pi_0$ consists of minuscule roots.
The hyperplanes $\{ v_0 = 0 \}$ replace the fixed points of the involutions (which in general have codimension greater than one) and allows us to slice the polytope $\mathcal{K}$ appropriately to obtain $\mathcal{L}$. 
%Below we illustrate this for type $A_2$.

Note that, in general, the Komrakov--Premet polytope $\K$ is not a simplex
(see the examples above). For simplicity, we would like to slice $\mathcal{K}$ through its vertices, so that the vertices of $\mathcal{L}$ are a subset of the vertices of $\K$. The choice of balanced roots does just that.

\begin{theoA}
\label{thm:B}
Let $\phi_0,\phi_1\in\Aut(D)$ be two involutions such that
$$
\Aut(\A)=\Omega\rtimes\left<\phi_0,\phi_1\right>
$$
and $v_0,v_1\in V$ be two compatible and balanced 
%minuscule 
roots
%normalized maximal minuscule roots
%vectors
for $\phi_0,\phi_1$, respectively. Then
$$
\mathcal{L}:=\{
x \in \mathcal{K} \mid 
(v_0, x)\geq 0,  
(v_1,x)\geq 0
\}
$$
is a fundamental polytope for $\Aut(\A)$ acting on $\A$, whose vertices are contained in the vertices of $\K$.
Using Bourbaki's ordering of the simple roots \cite{Bo68}, one can take
\[\left\{\begin{array}{ll}v_0=v_1=\sum_{i=1}^{\lfloor n/2\rfloor}(\alpha_i+\alpha_{n+1-i}) & \text{if $\Phi=A_n, n \geq 2$}, \\[.5em] 
v_0 = \alpha_1-\alpha_3, \quad v_1 = \alpha_3-\alpha_4 & \text{if $\Phi=D_4$}, \\[.5em] v_0=v_1=\alpha_{n-1}-\alpha_n & \text{if $\Phi=D_n, n \geq 5$}, \\[.5em] v_0=v_1=\alpha_1-\alpha_6 & \text{if $\Phi=E_6$}, \\[.5em] 
v_0=v_1=0  & \text{otherwise.}\end{array}\right.\]
\end{theoA}

The fact that $\Aut(\widetilde{D})$ is abstract Coxeter can be obtained by inspecting the classification of root systems: the importance of Theorem A is
to realize $\mathrm{Aut}(\widetilde{D})$ geometrically as $\Aut(\A)$. 
This realization is then used in Theorem B. 

The fact that we have a convenient description of the vertices \eqref{eq:kp-domain-vert} of the Komrakov--Premet fundamental domain is the reason we chose it over other possibilities, e.g.\ the Dirichlet fundamental domain for $\Omega$ (see Proposition \ref{prop:dirichilet}).
%
%Instead of starting with an euclidean root system $\Phi$ in an inner product space $V$, one can start just with root datum such as in \cite{Ga24} and from that produce a real vector space $V$ with $\Phi \subseteq V$ and $\Phi^\vee \subseteq V^*$, 
%$W$-invariant inner product in $V$ and $W^\vee$-invariant inner product in $V^*$ with unique corresponding $\O(V)$ and $\O(V^*)$, such that the Cartan isomorphism $\theta: \Gl(V) \to \Gl(V^*)$, $\phi \mapsto (\phi^*)^{-1}$, restricts to an isomorphism from $W$ to $W^\vee$, thus from $\O(V)$ to $\O(V^*)$. In this setting, $\Aut(\A) \leq \O(V^*)$, $W_{\mathrm{ext}} = P^\vee \rtimes W^\vee \leq \O(V^*)$ while $\Aut(\Phi) \leq \O(V)$. Thus one has to use Cartan isomorphism appropriately in the statement of Theorems A and B, for example:
%$\Aut(\A) = \Omega \rtimes \theta(\Aut(D))$.

The structure of this article is as follows. In Section 1 we state our main results: Theorems A and B, and introduce the setup and notation.  
In Sections 2, 3, 4 we prove Theorem A and in Section 5 we prove Theorem B. 
In the Appendix we discuss some related negative results on stratified centralizers and equivariant triangulations.

\vspace{6pt}
\noindent Acknowledgements:
We would like to thank C.\ Bonnaf\'e for discussions on the topic.
The first author would like to thank the Department of Mathematics at Friedrich-Alexander University Erlangen--Nuremberg for their warm hospitality during his sabbatical leave there between 2024-2025, and thank for the financial support of CAPES-Brasil-Código de Financiamento 001.

%%%%%%%%%%%%%%
\subsection{Preliminaries}
\label{sec:prelim}

A set of $n+1$ vectors $\{ p_0, p_1, \ldots, p_n\} \subseteq V$ is in general position if any of its proper subsets is l.i.\
A simplex $\sigma$ of $V$ is the convex hull of a set of $n+1$ vertices in general position, $\sigma = \conv( \{ p_i \} )$. There exists a unique point $q_*$ equidistant to the vertices of $\sigma$, it is the center of the sphere 
in which $\sigma$ in inscribed.
Also, there exists a unique point $p_*$ equidistant to the faces of 
$\sigma$, it is the center of the sphere inscribed in $\sigma$, which is tangent to each face of $\sigma$ at a unique point.
Each vertex $p_i$ of $\sigma$ has an opposite face $\sigma(p_i) = \conv( \{ p_j \mid j \neq i \} )$ with unit inward normal $\nu_i$. Furthermore, the unit normals $\nu_0, \nu_1, \ldots, \nu_n$ are in general position.

In the Euclidean group of rigid motions $\E(V) = V \rtimes \O(V)$, the vector space $V$ is identified with the subgroup of translations via $t_\gamma(H) = H + \gamma$ for $\gamma \in V$. Since $V$ is normal in $\E(V)$, the projection to the linear part 
\begin{equation}
\pi: \E(V) \to \O(V) 
\end{equation}
is a surjective homomorphism, equivariant w.r.t.\ the $\O(V)$-action by conjugation, 
that is
\begin{equation}
\label{eq:equivariance}
\pi( \phi \zeta \phi^{-1} ) = \phi \pi(\zeta) \phi^{-1}
\end{equation}
for $\zeta \in \E(V)$ and $\phi \in \O(V)$. 
 
An affine reflection group $W_{\mathrm{aff}}$ of $V$ is generated by the reflections on the hyperplanes of a locally finite and infinite collection $\mathcal{H}$ of affine hyperplanes of $V$ that is invariant w.r.t.\ reflection on each one of its hyperplanes (see Chap.~V of \cite{Bo68}). An alcove is then the closure of a connected component of the Stiefel diagram $V \setminus \bigcup_{H \in \mathcal{H}} H$. 

The group $W_{\mathrm{aff}}$ acts simply transitively on the set of alcoves
and any alcove is a fundamental domain for the action of $W_{\mathrm{aff}}$ in $V$. We assume from now on that this action is irreducible so that alcoves are simplices. 
Fix an alcove $\A$ through the origin, which we will call a fundamental alcove. 
Its inwards unit normals give rise to a Coxeter diagram $\widetilde{C}$.
We actually have more structure. Indeed, successive reflections on two successive paralell hyperplanes of $\mathcal{H}$, one of which meets $\A$, give rise to translations by a reduced and irreducible coroot system $\Phi^\vee \subseteq W_{\mathrm{aff}}$, with corresponding root system $\Phi$ and system of simple roots $\Pi$ that encode $\A$ as follows.  Fix simple roots $\Pi=\{\alpha_1,\dotsc,\alpha_n\}$ indexed by
$I:=\{1,\dotsc,n\}$ and the corresponding highest root
$\alpha_0=\sum_{i\in I} m_i\alpha_i$. We have that
\begin{equation}
\A:=\{x \in V \mid (\forall i\in I)\ (\alpha_i, x) \geq 0,~ (\alpha_0, x) \leq 1\}. 
%=\conv\left(\{0\}\cup\left\{\frac{\varpi^\vee_i}{n_i}\right\}_{i\in I}\right)\simeq\Delta^n.
\end{equation}
It follows that $\{-\alpha_0, \alpha_i \mid i \in I  \}$ are (nonunit) inward  normals to $\A$, giving rise to a Dynkin diagram $\widetilde{D}$, which projects onto $\widetilde{C}$. 

Clearly, $W_{\mathrm{aff}}$ contains the Weyl group $W$ of $\Phi$ and the translations by coroots $\alpha^\vee \in \Phi^\vee$. We actually have
\begin{equation}
W_{\mathrm{aff}} = Q^\vee \rtimes W, 
\end{equation}
where the coroot lattice $Q^\vee=\Z\Phi^\vee$ is $W$-invariant.
Consider the {\it fundamental coweights}
\[ \varpi^\vee_1,\dotsc,\varpi^\vee_n \in V \quad \mbox{ given by } \quad
  (\varpi^\vee_i,\alpha_j) = \delta_{ij},\]
spanning the {\it coweight lattice} 
\[ P^\vee = \{ \omega \in V \ |\  (\forall \alpha \in \Phi)\
    (\omega, \alpha) \in \Z\}.\]
In view of the integrality of the
Cartan--Killing integers $(\alpha^\vee, \beta)$, $\alpha, \beta \in \Phi$,
the coroot lattice $Q^\vee$ is a subgroup of $P^\vee$.
This gives rise to the {\it extended affine Weyl group} 
\begin{equation}
\label{eq:def-ext}
W_{\mathrm{ext}} := P^\vee \rtimes W
\end{equation}
Translation by $P^\vee$ leaves the Stiefel diagram of $\Phi$
invariant, so that $W_{\rm ext}$  also acts on the
set of alcoves of $\Phi$,
but not necessarily simply transitively: the stabilizer
\begin{equation} \label{eq:omega}
\Omega := \{ \zeta \in W_{\mathrm{ext}}  \mid \zeta(\A) = \A \}
\end{equation}
is called the {\it fundamental group} of $\Phi$. Since the action of $W_{\mathrm{aff}}$ on the alcoves is simply transitive, we get the decomposition
\begin{equation}
\label{eq:decompos-ext-omega}
W_{\mathrm{ext}} = W_{\mathrm{aff}} \rtimes \Omega
\end{equation}
From this discussion we get isomorphisms 
\begin{equation}
\label{eq:decompos-ext-omega-b}
P^\vee/Q^\vee
\leftarrow
W_{\mathrm{ext}}/W_{\mathrm{aff}} 
\rightarrow
\Omega.
\end{equation}
In particular, $\Omega$ is a finite abelian group
(with group structure given by Corollary \ref{corol4:euler} and Table \ref{table:omega}).

Now consider the automorphism group $\Aut(\Phi)$ of the root system $\Phi$ 
(see \cite{So22}).  
Since $\Phi$ is irreducible we have
\begin{equation}
\Aut(\Phi) = \{ \phi \in \O(V) \mid \phi(\Phi) = \Phi \} \leq \O(V)
\end{equation}
%(see Corollary 3.1 of \cite{So22}).  
The closed Weyl chamber w.r.t.\ to the system of simple roots $\Pi$ is the cone
\begin{equation}
\mathcal{C}:=\{x \in V \mid \forall i\in I,
(\alpha_i, x) \geq 0 \}
\end{equation}
with faces $\{ \alpha_i = 0 \}$, $i \in I$. Thus, the inward normal direction to its faces are given by $\Pi$, which
correspond to the nodes of the Dynkin diagram $D$ of $\Phi$. The automorphism group $\Aut(D)$ of \cancel{the} $D$ is the subgroup of $\Aut(\Phi)$ that stabilizes $\Pi \subseteq \Phi$.
As $W$ is generated by reflections on all root hyperplanes, 
it is a normal subgroup of $\Aut(\Phi)$. Since $W$ acts simply transitively on the systems of simple roots of $\Phi$ we get the decomposition 
%(see Proposition 6 of \cite{So22})
\begin{equation}
\label{eq:decompos-aut-1}
\Aut(\Phi) = W \rtimes \Aut(D)
\end{equation}

Recall that the inward normal directions of the fundamental alcove $\A$ are given by $-\alpha_0$ and the simple roots $\Pi$, which are the nodes of the corresponding affine Dynkin diagram $\widetilde{D}$ of $\Phi$. 
The automorphism group $\Aut(\widetilde{D})$ of $\widetilde{D}$ is the subgroup of $\Aut(\Phi)$ that stabilizes $\{-\alpha_0\} \cup \Pi \subseteq \Phi$.
Since $\Aut(D)$ stabilizes $\Pi$, it fixes the corresponding highst root $\alpha_0$ so that $\Aut(D) \leq \Aut(\widetilde{D})$. Intersecting \eqref{eq:decompos-aut-1} with $\Aut(\widetilde{D})$ we get the decomposition
\begin{equation}
\label{eq:decompos-aut-2}
\Aut(\widetilde{D}) = \widetilde{W} \rtimes \Aut(D), 
\end{equation}
where
\begin{equation}
\widetilde{W} := W \cap \Aut(\widetilde{D}). 
\end{equation}

%%%%%%%%%%%%%%
\section{Affine diagram automorphisms and minuscule roots}
\label{minuscule-roots}

In this section we describe how the affine diagram automorphisms and the minuscule roots are intimately related. 

Recall the highest root $\alpha_0 = \sum_{i>0} m_i \alpha_i$. We say that the root $\alpha_j$, or the coweight $\varpi^\vee_j$, is {\it minuscule} when the corresponding coefficient is $m_j=1$. Recall that \eqref{eq:decompos-aut-2} 
\[ \widetilde{W} := W \cap \Aut(\widetilde{D}) \]
acts simply transitively on the set of systems of simple roots
contained in $\{-\alpha_0\} \cup \Pi$.  Since any such system has $n$ elements, it is either $\Pi$ or of the form
\begin{equation}
\label{eq:simple-system-i}
\Pi_j = \{ -\alpha_0,\, \alpha_i:\, i \neq j \}
\end{equation}
The main result of this section is the following.

\begin{theo}
\label{thm:simple-system-i}
$\Pi_j$ is a system of simple roots for $\Phi$ if and only if the root $\alpha_j$ is minuscule. In this case, $\alpha_j$ is the lowest root of $\Pi_j$.
\end{theo}

\begin{proof}
$\Pi_j$ is clearly a linear independent set of vectors that spans $V$.
If $m_j = 1$, we claim that every root $\beta \in \Phi$ is $\Z$-spanned by $\Pi_j$ with coefficients of the same sign. In fact, w.l.o.g.\ $\beta$ is negative w.r.t.\ $\Pi$, thus it can be obtained bottom up from the lowest root $-\alpha_0$ w.r.t.\ $\Pi$ as
$$
\beta = -\alpha_0 + \sum_i c_i \alpha_i = 
-\alpha_0 + c_j\alpha_j + \sum_{i\neq j} c_i \alpha_i
=  (c_j -1)\alpha_j + \sum_{i\neq j} (c_i -m_i) \alpha_i
$$
for integers $c_i \geq 0$, where $c_i - m_i \leq 0$. On the rightmost side, the $\alpha_j$ coefficient of $\beta$ in the $\Pi$-basis is $c_j-1 \leq 0$, it follows that $c_j \in \{0,\, 1\}$.
If $c_j = 0$, then
$\beta = -\alpha_0 + \sum_i c_i \alpha_i$
has non-negative coefficients on the $\Pi_j$ basis.
If $c_j = 1$, then
$\beta = \sum_{i\neq j} (c_i - m_i)\alpha_i$
has non-positive coefficients on the $\Pi_j$ basis.
In both cases the claim hold, thus $\Pi_j$ is a system of simple roots.

Reciprocally, if $\Pi_j$ is a system of simple roots then there exists unique integer coefficients $c_i$ of the same sign such that
\begin{eqnarray*}
\alpha_j = 
c_0(-\alpha_0) + \sum_{i \neq j} c_i \alpha_i
\end{eqnarray*}
Since in the rightmost side the coefficient of $\alpha_j$ in the $\Pi$-basis is $-c_0m_j$, it follows that $-c_0m_j=1$ so that 
$-c_0 = m_j = 1$, thus $\alpha_j$ is minuscule.

If $\Pi_j$ is a system of simple roots we claim $\alpha_j$ is its lowest root. In fact, 
subtracting $\alpha_i$, $i \neq j$, from $\alpha_j$ amounts to $\alpha_j - \alpha_i$, which is never a root, since $\Pi$ is a simple system. Furthermore, subtracting $-\alpha_0$ from $\alpha_j$ amounts to $\alpha_j + \alpha_0$, which is never a root, since $\alpha_0$ is the highest root of the simple system $\Pi$. 
\end{proof}

%Since the Weyl group $W$ of the root system $\Delta$ acts simply transitively on its systems of simple roots we get that $W \cap \Aut(D) = 1$ and also the following.
%Recall that $\widetilde{W} = W \cap \Aut(\widetilde{D})$. 

\begin{cor}
\label{corol1:euler}
Let $\alpha_j$ be a minuscule root. Then there exists a unique
$v_j \in W$ such that \begin{equation}
\label{eq:def-v-omega}
v_j \Pi = \Pi_j
\end{equation}
or, equivalently, such that $v_j \in \widetilde{W}$ and 
$v_j( -\alpha_0 ) = \alpha_j$.
\end{cor}

\begin{proof}
Since $\alpha_j$ is minuscule, $\Pi_j$ is a system of simple roots, by Theorem \ref{thm:simple-system-i}. The uniqueness of $v_j$ then comes from the simple transitivity of $W$ on the set of systems of simple roots of $\Phi$.
Now let $v_j \in W$ such that $v_j \Pi = \Pi_j$. Then $v_j$ maps the lowest root of $\Pi$ to the lowest root of $\Pi_j$ so that, by Theorem \ref{thm:simple-system-i}, $v_j( -\alpha_0 ) = \alpha_j$, thus $v_j \in \widetilde{W}$.
Reciprocally, if $v_j \in  \Aut(\widetilde{D})$ then
$$
v_j( \Pi \cup \{-\alpha_0\} ) = \Pi \cup \{-\alpha_0\}
$$
If $v_j( -\alpha_0 ) = \alpha_j$ then we can remove 
$\{-\alpha_0\}$ from the leftmost parenthesis and $\{\alpha_j\}$ from the rightmost parenthesis to get $v_j \Pi = \Pi_j$, as claimed.
\end{proof}

The next result gives a topological characterization of the minuscule roots in terms of the affine Dynkin diagram.

% OK \red{KH: corollary rewritten} 
\begin{cor}
\label{corol2:euler}
For a simple root $\alpha_j$, the following are equivalent:
\begin{itemize}
\item[\rm(a)] $\alpha_j$ is minuscule. 
\item[\rm(b)] It can be mapped to $-\alpha_0$ by an automorphism of the affine diagram $\widetilde{D}$. 
\item[\rm(c)] The diagram $\widetilde{D} \setminus \{ \alpha_j\}$ is
isomorphic to~$D$.
\end{itemize}
\end{cor}
\begin{proof}
  If $\alpha_j$ is minuscule, then Corollary \ref{corol1:euler}
  implies that $f = v_j^{-1} \in \mathrm{Aut}(\widetilde{D})$ is such that $f(\alpha_j) = -\alpha_0$. It restricts to a diagram isomorphism $f: \widetilde{D} \setminus \{\alpha_j\} \to \widetilde{D} \setminus \{-\alpha_0\} = D$, so (a) implies (b), and (b) is equivalent to (c). 
  To see that (c) implies (a), assume that
  there exists a diagram isomorphism $f: \widetilde{D} \setminus \{\alpha_j\} \to D$. Then $\widetilde{D} \setminus \{\alpha_j\}$ is a
  system of simple roots for $\Phi$.  Since the roots of this diagram are $\Pi_j$, 
  Theorem \ref{thm:simple-system-i} implies that $\alpha_j$ is minuscule.
\end{proof}

\begin{cor}
\label{corol3:euler}
We have that
\begin{equation}
\widetilde{W} = \{ 1\} \cup \{ v_j:\, \alpha_j \text{ is minuscule}\, \}
\end{equation}
which acts simply transitively on the subset of roots $\{ -\alpha_0,\, \alpha_j:\, \alpha_j \text{ is minuscule} \}$.
%of inside pointing normals of the fundamental alcove $\A$.
\end{cor}
\begin{proof}
Corollary \ref{corol1:euler} implies the inclusion ``$\supseteq$'' and also the transitivity, since $-\alpha_0$ can be mapped to every other minuscule $\alpha_j$.
For the inclusion ``$\subseteq$'', let $w \in \widetilde{W}$. If $w(-\alpha_0) = -\alpha_0$ then $w \in \Aut(D) \cap W = 1$, which also proves the simple transitivity. Otherwise $w(-\alpha_0) = \alpha_j$, for some $i>0$, and, since $w \in \widetilde{W}$, Corollary \ref{corol1:euler} implies that $w = v_j$.
\end{proof}

Next, we show that $\Aut(\widetilde{D})$ acts on the set of roots with any fixed $m_i$ coefficient, that is, $\Aut(\widetilde{D})$ preserves the $m_i$ coefficients, 
even thought in general it may not be transitive anymore.

\begin{prop}
\label{lemma1}
Let $\zeta \in \mathrm{Aut}(D)$. If $\zeta(\alpha_j) = \alpha_k$ then $\zeta(\varpi^\vee_j) = \varpi^\vee_k$ and $m_j = m_k$. Furthermore, $\zeta(\alpha_0) = \alpha_0$.
\end{prop}

\begin{proof}
If $\zeta(\alpha_j) = \alpha_{\pi(j)}$, where $\pi$ is a permutation of $\{1,\ldots,n\}$, then $\zeta^{-1}(\alpha_k) = \alpha_{\pi^{-1}(k)}$.
Now
$( \zeta(\varpi^\vee_j), \alpha_k ) 
= ( \varpi^\vee_j, \zeta^{-1}(\alpha_k) )
= ( \varpi^\vee_j, \alpha_{\pi^{-1}(k)} )
= \delta_{j,\pi^{-1}(k)}
= \delta_{\pi(j), k}$, which proves that $\zeta(\varpi^\vee_j) = \varpi^\vee_{\pi(j)}$.
For the second claim, $\zeta$ maps the system of simple roots $\Pi$ into itself, thus fixes the corresponding highest roots, so that $\zeta(\alpha_0) = \alpha_0$. Writing $\alpha_0 = \sum_{i>0} m_{\pi(j)} \alpha_{\pi(j)}$ we get from $\zeta(\alpha_0) = \alpha_0$ that
$$
\sum_{i>0} m_j \alpha_{\pi(j)} = \sum_{i>0} m_{\pi(j)} \alpha_{\pi(j)}
$$
so that $m_j = m_{\pi(j)}$, as claimed.
\end{proof}

\begin{prop}
\label{lemma2}
Let $\zeta \in \mathrm{Aut}(\widetilde{D})$.
If $\zeta(\alpha_j) = \alpha_k$ for $j, k > 0$, then $m_j = m_k$.
%Furthermore, if $\zeta(-\alpha_0) = \alpha_i$ for some $i>0$, then $m_i = 1$.
\end{prop}
\begin{proof}
If $\zeta(-\alpha_0) = -\alpha_0$ then $\zeta$ restricts to an automorphism of $D$ so that Proposition~\ref{lemma1} implies the conclusion about $m_j = m_k$. If, on the other hand, $\zeta(-\alpha_0) = \alpha_i$ for some $i>0$, then there exists $k>0$ such that $\zeta(\alpha_k) = -\alpha_0$ and, for $j > 0$ with $j \neq k$, we have $\zeta(\alpha_j) = \alpha_{\pi(j)}$ with $\pi(j) \neq i$. 
Thus, $\pi$ is a bijection from $\{j=1,\dots,n: j \neq k \}$ to $\{j=1,\dots,n: j \neq i \}$. %From $\alpha_0 = \sum_j m_j \alpha_j$
It follows that
$$
\zeta(- \alpha_0) = - \sum_j m_j \zeta(\alpha_j) = m_k \alpha_0 - \sum_{j\neq k} m_j \alpha_{\pi(j)}.
$$
Substituting $j = \pi^{-1}(r)$ we get
$$
\zeta(- \alpha_0) = m_k \alpha_0 - \sum_{r\neq i} m_{\pi^{-1}(r)} \alpha_r
$$
Using that $\alpha_0 = m_i \alpha_i + \sum_{r\neq i} m_r \alpha_r$, we get
\begin{eqnarray*}
\zeta(- \alpha_0)  
= m_k m_i \alpha_i + \sum_{r\neq i} (m_k m_r - m_{\pi^{-1}(r)}) \alpha_r 
= \alpha_i
\end{eqnarray*}
so that
$m_i = m_k = 1$, and thus
$m_r - m_{\pi^{-1}(r)} = 0$, for $r \neq i$.  Substituting back $r = \pi(j)$, we get $m_j = m_{\pi(j)}$ for $j \neq k$, as claimed.
\end{proof}

We finish this section by obtaining, with an alternative proof, the representation of $v_j \in \widetilde{W}$ as a product of principal involutions (see VI.\S 2.3, Prop.~6 of \cite{Bo68}). Let $w^-$ be the principal involution of $W$, that is, the unique element $w^- \in W$ such that $w^-\Pi = -\Pi$, and note that $(w^-)^2=1$.
Let $W_j \leq W$ be Weyl group of
  the root system generated by the simple roots in
  $\Pi\setminus \{ \alpha_j \}$, and $w^-_j$ its principal involution, so that $w^-_j(\Pi\setminus \{ \alpha_j \}) = -(\Pi\setminus \{ \alpha_j \})$ and $(w^-_j)^2=1$.

\begin{prop}
\label{prop:prod-invol}
$v_j = w_j^- w^-$
\end{prop}
\begin{proof}
First we prove that $w_j^- \alpha_0 = \alpha_j$. Since the roots in $\Pi\setminus \{ \alpha_j \}$ annihilate $\varpi^\vee_j$, it follows that $W_j$ centralizes $\varpi^\vee_j$, thus the coefficient of $\alpha_j$ in $w_j^- \alpha_j$ is 
\[ (w_j^- \alpha_j, \varpi^\vee_j) = 
(\alpha_j, w_j^- \varpi^\vee_j) =
(\alpha_j, \varpi^\vee_j) = 1.\]
In particular $w_j^- \alpha_j$ is a positive root so that 
$w_j^-(\alpha_j) = \alpha_j + \sum_{i\neq j} c_i \alpha_i$,
with non-negative integer coefficients $c_i$. 
Since $w_j^-$ sends the $\alpha_i$, $i\neq j$, to negative simple roots, we have that $w^-_j \alpha_i = -\alpha_{\pi(j)}$ where $\pi$ is an involution of $\{j:\, i \neq j\}$. It follows that
\begin{equation}
\label{eq:proof-invol}
w_j^- \alpha_0 = w_j^-\Big( \alpha_j + \sum_{i\neq j} m_i \alpha_i\Big)
= \alpha_j + \sum_{i\neq j} c_i \alpha_i - \sum_{i\neq j} m_i \alpha_{\pi(i)} =
\alpha_j + \sum_{i\neq j} (c_i - m_{\pi(i)}) \alpha_i
\end{equation}
is a positive root, where in the next to the last equation in \eqref{eq:proof-invol} we changed the last summation variable from $i$ to $\pi(i)$. It follows that $m_{\pi(i)} \leq c_i$ so that
$$
m_{\pi(i)} \leq c_i \leq m_i
$$
As $\pi$ is an involution and $m_{\pi(i)} \leq m_i$ holds for all $i$,
we have $m_{\pi(i)} = m_i$, so that  $m_{\pi(i)} = c_i $ and \eqref{eq:proof-invol} implies $w^-_j\alpha_0 = \alpha_j$.
%Indeed, let $f(i) = m_i - m_{\pi(i)} \geq 0$ for all $i \neq j$.  Since $\pi$ is an involution we have $f(\pi(i) ) = m_{\pi(i)} - m_i = - f(i) \leq 0$.  It follows that $f(\pi(i) ) = 0$ and thus $f(i) = 0$ for all $i \neq j$, as claimed.  This proves that $w_j^- \alpha_0 = \alpha_j$.
Since
$$
\Pi_j = (\Pi\setminus \{ \alpha_j \}) \cup \{-\alpha_0\}
$$
this implies that
$$
w_j^-\Pi_j = -(\Pi\setminus \{ \alpha_j \}) \cup \{-\alpha_j\} = -\Pi, 
$$
so that $w^-w_j^-\Pi_j = \Pi$. The characterization of $v_j$ in Corollary \ref{corol1:euler} then implies that 
$v_j = (w^-w_j^-)^{-1} = w_j^-w^-$, as claimed.
%Indeed, partition $\{j:\, i \neq i\} = J^+_0 \cup \cdots \cup J_k$ in such a way that $m_i$ is constant and equal to $\mu_r$ in each $J_r$ and, for $r \leq s$ we have $\mu_r \leq \mu_s$.  It follows that $\mu_n$ is $n$-th smallest value of $m_i$, $i \neq j$. Thus for $j \in J^+_0$, $\mu_1 \leq m_{\pi(i)} \leq m_i = \mu_1$ implies $m_{\pi(i)} = m_i = \mu_1$, in particular $J^+_0$ is $\pi$-invariant. Then for $j \in J^-_0_0$ we have $\mu_1 < m_{\pi(i)} \leq m_i = \mu_2$ so that $m_{\pi(i)} = m_i = \mu_2$, in particular $J^-_0$ is $\pi$-invariant. And so on, until $m_{\pi(i)} = m_i = \mu_i$ for $j \in J_k$, which proves our claim.
\end{proof}

\begin{rem}
Proposition \ref{prop:prod-invol} can be used to read the $v_j$ from the Dynkin diagram as follows.
Write $-w^- \in \Aut(D)$ and $-w_j^- \in \Aut(D \setminus \{\alpha_j\})$ as permutation of roots. The product permutation $v_j = (-w_j^-)(-w^-)$ is not in $\Aut(D)$, since it does not take $\alpha_j$ to a simple root. But we know beforehand that $v_j(-\alpha_0) = \alpha_j$, thus we can complete the permutation for $v_j$ by taking $0$ to $j$ so that it becomes a permutation in $\Aut(\widetilde{D}) \cap W = \widetilde{W}$.
The following table describes the group structure of $\widetilde{W}$ 
(see \cite{Ga23}). 

\begin{figure}[h!]
\includegraphics[width=15.5cm]{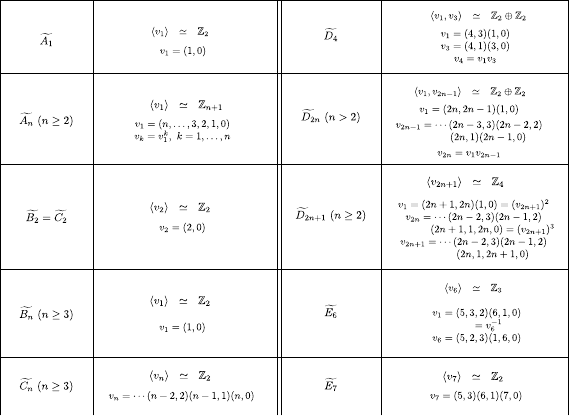}
\captionof{table}{Group structure of $\widetilde{W}$
with permutations of roots $v_j$ in cycle notation and
indices on the corresponding affine Dynkin diagrams in Table \ref{table:thm1}.
}
\label{table:omega}
\end{figure}

\end{rem}

%\begin{cor}
%\label{cor:inverse-omega}
%We have that
%$w^-(\varepsilon_i) = - \varepsilon_j$,
%where $\omega_i^{-1} = \omega_j$.
%\end{cor}
%\begin{proof}
%In fact, by the previous Proposition, conjugation by $(-w^-)$ sends $v_i$ to $v_i^{-1}$. Thus, $(-w^-) \omega_i (-w^-)^{-1} = t_{-w^-(\varepsilon_i)} v_i^{-1} = \omega_i^{-1} = \omega_j  = t_{\varepsilon_j} v_j$.
%\end{proof}

\section{Fundamental group and minuscule coweights}

Recall the fundamental group $\Omega$ of $W_{\mathrm{aff}}$, defined in \eqref{eq:omega} as the stabilizer of $\A$ in $W_{\mathrm{ext}}$.  In this section we use results of the previous section to give an alternative proof of a well known description of $\Omega$ (see Proposition VI.\S 2.3, 6 of \cite{Bo68}).  

\begin{theo}
\label{theo:alcove-pij}
The fundamental alcove w.r.t.\ the system of simple roots $\Pi_j$ is given by
\begin{equation}
v_j \A = \A - \varpi^\vee_j
\end{equation}
\end{theo}
\begin{proof}
Since by \eqref{eq:def-v-omega}, we have $v_j \Pi = \Pi_j$, it follows immediately that $v_j \A$ is the fundamental alcove w.r.t.\ $\Pi_j$. On the other hand, for $x \in \A$ we have
\begin{eqnarray*}
\alpha_i(x - \varpi^\vee_j) & = & \alpha_i(x) \geq 0, \quad i \neq j \\ 
\alpha_0(x - \varpi^\vee_j) & = & \alpha_0(x) - 1  \leq 0\phantom{-} \! 
\quad \Longrightarrow -\alpha_0(x - \varpi^\vee_j) \geq 0 \\
\alpha_j(x - \varpi^\vee_j) & = & \alpha_j(x) - 1 \geq -1 \quad \Longrightarrow -\alpha_j(x - \varpi^\vee_j) \leq 1
\end{eqnarray*}
which is the fundamental alcove w.r.t.\ $\Pi_j = \{ -\alpha_0,\, \alpha_i:\, i \neq j \}$ since, by Theorem \ref{thm:simple-system-i}, its highest root is $-\alpha_j$. 
\end{proof}

The vertices of the fundamental alcove $\A$ are given by
\begin{equation}
\label{eq:vertices-A}
\{0,\, \varpi^\vee_i/m_i:\, i=1,\ldots,n\}
\end{equation}
Thus, its intersection with the coweight lattice $P^\vee$ is given by
\begin{equation}
\label{eq:euler vertices}
\A \cap P^\vee = \{ 0,\, \varpi^\vee_j:\, \varpi^\vee_j \text{ is minuscule}\, \}
\end{equation}
the {\it minuscule vertices of $\A$}. 
For $\varpi^\vee_j \in \A \cap P^\vee$, 
Theorem \ref{theo:alcove-pij} implies that $v_j \A + \varpi^\vee_j = \A$, so that 
\begin{equation}
\omega_j := t_{\varpi^\vee_j} v_j \,\, \in \Omega
\end{equation}
Reciprocally, we have the following.

\begin{cor}
\label{propos:omega-elements}
We have that
\begin{equation}
\label{eq:omega-elements}
\Omega = \{ 1,\, \omega_j:\, \alpha_j \text{ is minuscule}\, \}
\end{equation}
which acts simply transitively on the subset of minuscule vertices $\A \cap P^\vee$. 
\end{cor}

\begin{proof}
Let $\zeta = t_\gamma w \in W_{\mathrm{ext}}$, with $\gamma \in P^\vee$ and $w \in W$, be such that $\zeta \in \Omega$, so that
$$ \A = \zeta(\A) = w\A + \gamma. $$
Since $0 \in w\A$, it follows that $\gamma \in \A$. By \eqref{eq:euler vertices} it follows that $\gamma$ is either $0$ or a minuscule vertex $\varpi^\vee_j$. 
We will use that $W_{\mathrm{aff}}$, thus $W$, acts simply transitively on the set of alcoves.
If $\gamma=0$ then $w\A = \A$, which implies $w = 1$, so that $\zeta = 1$. 
If, otherwise, $\gamma = \varpi^\vee_j$, then $w\A = \A - \varpi^\vee_j$ so that Theorem \ref{theo:alcove-pij} implies that $w = v_j$, so that $\zeta = \omega_j$.

Now $\Omega$ stabilizes both $\A$ (by definition) and $P^\vee$ (which is $W_{\mathrm{ext}}$-invariant), thus it acts on the set $\A\cap P^\vee$ of minuscule vertices. Since $\omega_j(0) = \varpi^\vee_j$, it follows that this action is transitive and that the stabilizer of $0$ in $\Omega$ is trivial, thus it is simply transitive.
\end{proof}

\begin{lem}
\label{propos:aut-D-subgrp}
The subgroup of $\Aut(\A)$ generated by $\Omega$ and $\Aut(D)$ is a
semidirect product
\begin{equation}
\label{eq:aut-D-subgrp}
\Omega \rtimes \Aut(D) \leq \Aut(\A)
\end{equation}
\end{lem}

\begin{proof}
  To see that $\Omega \cap \Aut(D) = \{1\}$, note that $\Aut(D)$
  fixes $-\alpha_0$, while, by Corollary~\ref{corol3:euler}, $\Omega$ acts simply transitively on $\{ -\alpha_0, \, \alpha_j: j \in J \}$.  
  To show that $\Aut(D)$ normalizes $\Omega$, first note that $\Aut(D)$ is a subgroup of $\Aut(\Phi)$, which normalizes $W$.
  By Proposition~\ref{lemma1}, $\Aut(D)$ permutes the coweights, thus stabilizes the coweight lattice $P^\vee$. Thus, $\Aut(D)$ normalizes $W_{\mathrm{ext}} = P^\vee \rtimes W$. Since $\Aut(D)$ stabilizes $\A$, it follows that it normalizes $\Omega$, by its definition~\eqref{eq:omega}.
\end{proof}

By \eqref{eq:decompos-aut-2} and \eqref{eq:aut-D-subgrp}, 
$\Aut(D)$ acts by conjugation on $\widetilde{W}$ and $\Omega$. The equivariance \eqref{eq:equivariance} then implies the following.
\begin{cor}
\label{corol4:euler}
Projection to the linear part gives an isomorphism
\begin{equation}
\label{eq:restr-pi}
\pi: \Omega \to \widetilde{W}
\end{equation}
which is equivariant w.r.t.\ the $\Aut(D)$-action by conjugation.
\end{cor}
\begin{proof}
The fact that it is an isomorphism follows from $\pi(\omega_j) = \pi(t_{\varpi^\vee_j} v_j) = v_j$, and the descriptions of $\widetilde{W}$ in Corollary \ref{corol3:euler} and of $\Omega$ in Corollary \ref{propos:omega-elements}. 
\end{proof}

We finish this section by obtaining, with an alternative proof, the Dirichlet domain of $\Omega$ in $\A$ (see Equation (4.5) of \cite{Ya08}). Given a subgroup $G \subseteq \E(V)$ acting on $\A$ by affine isometries, its Dirichlet domain around the origin is given by 
\begin{equation}
\mathcal{D}(G) := \{ x \in \A:\, \|x\| \leq \|gx\|
\quad \text{for all } g \in G \}
\end{equation}
and is a fundamental domain for the $G$-action in $\A$ if the isotropy of $0$ is trivial.
 
\begin{prop}
\label{prop:dirichilet}
$\mathcal{D}(\Omega) =  \{ x \in \mathcal{A} \mid (\forall j \in J)\ \|x\| \leq \|x - \varpi^\vee_j\| \}$ 
\end{prop}
\begin{proof}
For $j \in J$ we have
\[ \|\omega_i x\| = \|v_ix + \varpi^\vee_i\| = \|x + v_i^{-1}\varpi^\vee_i\|\] 
Let $\omega_i^{-1} = \omega_j$. From
$$
1 = \omega_i \omega_j 
= t_{\varpi^\vee_i} v_i t_{\varpi^\vee_j} v_j
= t_{\varpi^\vee_i} t_{v_i\varpi^\vee_j} v_i v_j
= t_{\varpi^\vee_i + v_i\varpi^\vee_j} v_i v_j 
$$
it follows that $\varpi^\vee_i + v_i\varpi^\vee_j = 0$, so that $ v_i^{-1}\varpi^\vee_i = -\varpi^\vee_j$. Thus $\|\omega_i x\| = \|x - \varpi^\vee_j\|$ and the result follows by \eqref{eq:omega-elements}.
\end{proof}

\section{Automorphisms of the fundamental alcove}

Let $\sigma$ be a simplex and consider the simplex $\nu$ whose vertices are the inward unit normals to the faces of $\sigma$.

\begin{prop}
\label{prop:iso-angle-angle}
Projection to the linear part yields the isomorphism
$\pi: \Aut(\sigma) \to \Aut(\nu)$,
where $\Aut(\nu) = \{ \phi \in \O(V) \mid \phi(\nu) = \nu \}$.
\end{prop}
%For the injectivity, if $\pi(\zeta) = 1$ then $\zeta = t_\gamma$ is a translation.
%But the only translation that stabilizes a bounded set is $\gamma = 0$.  
%It follows that $\gamma = 0$, thus $\zeta = 1$, as claimed.  
%For the surjectivity, 
\begin{proof}
First, let $q^*$ be the unique point equidistant to the vertices of $\nu$, it follows that an isometry of $\nu$ fixes $q^*$. Since the vertices are unit normals, the origin $0$ is equidistant to the vertices, it follows that $q^* = 0$. Thus, every isometry of $\nu$ fixes $0$, so that $\Aut(\nu) \leq \O(V)$, which proves the second claim.

For the first claim, let $p^*$ be the unique point equidistant to the faces of $\sigma$: it is the center of the sphere inscribed in $\sigma$, which is tangent to each one of its faces at a unique point. 
%\red{KH: Give reference for existence.}
\begin{figure}[h!]
\def\svgwidth{9cm}
%% Creator: Inkscape 1.0.1 (c497b03c, 2020-09-10), www.inkscape.org
%% Accompanies image file 'incenter.pdf' (pdf, eps, ps)
\begingroup%
  \makeatletter%
  \providecommand\color[2][]{%
    \errmessage{(Inkscape) Color is used for the text in Inkscape, but the package 'color.sty' is not loaded}%
    \renewcommand\color[2][]{}%
  }%
  \providecommand\transparent[1]{%
    \errmessage{(Inkscape) Transparency is used (non-zero) for the text in Inkscape, but the package 'transparent.sty' is not loaded}%
    \renewcommand\transparent[1]{}%
  }%
  \providecommand\rotatebox[2]{#2}%
  \newcommand*\fsize{\dimexpr\f@size pt\relax}%
  \newcommand*\lineheight[1]{\fontsize{\fsize}{#1\fsize}\selectfont}%
  \ifx\svgwidth\undefined%
    \setlength{\unitlength}{237.76484482bp}%
    \ifx\svgscale\undefined%
      \relax%
    \else%
      \setlength{\unitlength}{\unitlength * \real{\svgscale}}%
    \fi%
  \else%
    \setlength{\unitlength}{\svgwidth}%
  \fi%
  \global\let\svgwidth\undefined%
  \global\let\svgscale\undefined%
  \makeatother%
  \begin{picture}(1,0.29675579)%
    \lineheight{1}%
    \setlength\tabcolsep{0pt}%
    \put(0,0){\includegraphics[width=\unitlength,page=1]{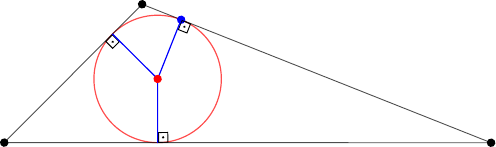}}%
    \put(0.34426389,0.12831091){\color[rgb]{1,0,0}\rotatebox{0.72471113}{\makebox(0,0)[lt]{\lineheight{1.25}\smash{\begin{tabular}[t]{l}$p^*$\end{tabular}}}}}%
    \put(0,0){\includegraphics[width=\unitlength,page=2]{incenter.pdf}}%
  \end{picture}%
\endgroup%

\end{figure}
It follows that an isometry of $\sigma$ fixes $p^*$ and permutes its points of intersection with the faces of $\sigma$, where the directed segment from each point of intersection to the center of the sphere has the direction of the inward normal to the corresponding face.
It follows that $\phi \in \Aut(\nu)$ determines uniquely $\zeta \in \Aut(\sigma)$ such that $\pi(\zeta) = \phi$, given by $\zeta(x) = \phi(x - p^*) + p^*$.
\end{proof}

%Note that the duality between {\em fundamental coweights and simple roots} is akin to the duality between {\em vertices and faces}. In fact, the results of this section suggests that, geometrically, $\Aut(D)$ permutes the vertices of $\A$, while $\Aut(\widetilde{D})$ permutes the faces of $\A$, that is, their normal vectors.
%
Recall the projection of a Dynkin diagram onto its corresponding Coxeter diagram. Geometrically, it amounts to unit normalization of the root vectors that realize the diagram.  Indeed, the Dynkin diagram encodes the angles and relative lengths between the root vectors of a simple system so that, normalizing them, we end up with only the angle data, which is precisely what is encoded by the corresponding Coxeter diagram.

Note that projecting the Dynkin diagram $D$ of the root system $\Phi$ we get the Coxeter diagram $C$ of its Weyl group $W$, where clearly $\Aut(D) \subseteq \Aut(C)$.

\begin{prop}
\label{propos:AutC}
$\Aut(\mathcal{C}) = \Aut(C)$.
\end{prop}

\begin{proof}
Let $\phi \in \Aut(C)$, then $\phi \in \O(V)$ permutes the unit normalized roots in $\Pi$. Thus $\phi(\alpha_i) = \lambda_{ij} \alpha_j$, with
$\lambda_{ij} = \|\alpha_i\|/\|\alpha_j\| >0$,
%\red{KH: now we have absolute values everywhere else. We should be consistent.}
so that $\phi(\mathcal{C}) = \mathcal{C}$ and then $\phi \in \Aut(\mathcal{C})$. Reciprocally, let $\zeta \in \Aut(\mathcal{C})$. Since $\zeta( \mathcal{C} ) = \mathcal{C}$, it permutes the top and lower dimensional faces of $\mathcal{C}$. In particular $\zeta(0) = (0)$, so that $\zeta \in \O(V)$
permutes the unit inward normals to the faces of $\mathcal{C}$, given precisely by the unit normalization of the roots in the  Dynkin diagram~$D$. Thus $\zeta$ permutes the vertices of the realization of the corresponding Coxeter diagram $C$ preserving their angles, so that $\zeta \in \Aut(C)$, as claimed.
\end{proof}

\begin{prop}
\label{propos:AutA-projection}
Projection to the linear part yields the isomorphism
\begin{equation}
\label{eq:proj-A}
\pi: \Aut(\A) \to \Aut(\widetilde{D})
\end{equation}
which restricts to the identity on $\Aut(D)$.
\end{prop}

\begin{proof}
Since the unit inward normals to the faces of $\A$ are given precisely by the unit normalization of the roots in the affine Dynkin diagram $\widetilde{D}$, thus precisely by the vertices of the realization of the corresponding Coxeter diagram $\widetilde{C}$, Proposition \ref{prop:iso-angle-angle} gives us the isomorphism
\[
\pi: \Aut(\A) \to \Aut(\widetilde{C})
\]
Projecting now the affine Dynkin diagram $\widetilde{D}$
to the Coxeter diagram $\widetilde{C}$ of its affine Weyl group $W_{\mathrm{aff}}$, the inclusion
$\Aut(\widetilde{D}) \subseteq \Aut(\widetilde{C})$
becomes an isomorphism, since the affine Dynkin diagram of a reduced root system is completely determined by its associated Coxeter diagram, as can be seen by inspection of the first column of Table \ref{table:thm1} (see also Remark in p.211, Chap.~VI, \S4.3 of \cite{Bo68}). It follows that
\[ 
\Aut(\widetilde{C}) = \Aut(\widetilde{D}),
\] 
which gives us the isomorphism \eqref{eq:proj-A}, as claimed. It clearly restricts to the identity in $\Aut(D)$, since $\Aut(D)$ is linear.
\end{proof}

By Corollary \ref{corol4:euler}, the restriction $\pi: \Omega \to \widetilde{W}$ is an isomorphism equivariant w.r.t.\ the $\Aut(D)$-action by conjugation.
Denote its inverse by
\begin{eqnarray}
\label{eq:restr-pi-inverse}
q: \widetilde{W} &\to& \Omega
\end{eqnarray}
which is again $\Aut(D)$-equivariant.
The decomposition \eqref{eq:decompos-aut-2} then suggests the following extension of $q$ to $\Aut(\widetilde{D})$.
Let
\begin{eqnarray*}
A & \rtimes_\varphi & B \\
\alpha \downarrow \,& &  \,\downarrow \beta \\
A' & \rtimes_{\varphi'} & B'
\end{eqnarray*}
be a square of group homomorphisms
$\alpha: A \to A'$, $\beta: B \to B'$, and semidirect products with actions $\varphi: B \to \Aut(A)$, $\varphi': B' \to \Aut(A')$.  Consider the componentwise map
$$
\alpha \rtimes \beta: A \rtimes_\varphi B \to A' \rtimes_{\varphi'} B',
\qquad
ab \mapsto \alpha(a)\beta(b)
$$
$a \in A$, $b \in B$. If $\alpha$ is equivariant in the sense that it intertwines the $\varphi$ and $\varphi'$ actions, then $\alpha \rtimes \beta$ is a homomorphism with kernel $\Ker \alpha \rtimes \Ker \beta$ and image $\Ima \alpha \rtimes \Ima \beta$.

\begin{theo}
\label{theo:inversa-A} 
The inverse to
$\pi: \Aut(\A) \to \Aut(\widetilde{D})$ is given by 
%the componentwise map 
$\theta = q \rtimes \id$, that is
\begin{eqnarray}
\label{eq:inversa-A} 
\theta: \Aut(\widetilde{D}) & \to & \Omega \rtimes \Aut(D) \\
\zeta \phi  & \mapsto & q(\zeta) \phi \nonumber
\end{eqnarray}
where $\zeta \in \widetilde{W}$ and $\phi \in \Aut(D)$,
It follows that $\theta$ is the identity on $\Aut(D)$ and $\pi$ is an isomorphism.
In particular, we get the semi-direct product decomposition
\begin{equation}
\label{eq:decompos-A} 
\Aut(\A) = \Omega \rtimes \Aut(D)
\end{equation}
\end{theo}

\begin{proof}
By the $\Aut(D)$-equivariance of $q$ of Corollary \ref{corol4:euler} and by the comments preceding the theorem, it follows that $\theta = q \rtimes \id$ is a homomorphism.
Since $q$ is an isomorphism, $\theta$ is also an isomorphism.
We claim that $\pi \circ \theta = \id$. Indeed, we have that
$$
\pi(\theta(\zeta\phi)) = \pi(q(\zeta)\phi) = \pi(q(\zeta))\pi(\phi) = \zeta \phi
$$
since $q$ is the inverse of $\pi$ restricted to $\Omega$, and $\pi$ restricted to $\Aut(D)$ is the identity.  It follows that $\pi$ is surjective with right inverse $\theta$.
By Proposition \ref{propos:AutA-projection}, $\pi$ was already injective, so that $\pi$ is isomorphism with inverse $\theta$. Decomposition \eqref{eq:decompos-A} then follows. 
\end{proof}

%It follows that projection to the linear part $\pi: \Aut(\A) \to \Aut(\widetilde{D})$ is an isomorphism.
%that extends $q$ and is the identity on $\Aut(D)$.

\begin{rem} Theorem~\ref{theo:inversa-A}  implies 
  that $\Aut(\widetilde{D})$ is a quotient of $\Aut(\A)$, while
$\Aut(D)$ is a subgroup of $\Aut(\A)$
(the subgroup that fixes the origin). 
\end{rem}

In order to obtain explicit expressions for the inverse isomorphisms $q$ and $\theta$, denote $v_0 := 1 \in \widetilde{W}$, $\omega_0 := 1 \in \Omega$ and $\varpi^\vee_0 := 0 \in V$. Then an explicit expression for $q$ is $q(v_j) = \omega_j$, $j \in J \cup \{0\}$ and we have the following.

\begin{prop}
\label{prop:inversa-A-explicit} 
An explicit expression for the inverse isomorphism $\theta$ is
\begin{eqnarray}
\label{eq:inversa-A-explicit} 
\theta: \Aut(\widetilde{D}) & \to & \Aut(\A)  \nonumber \\
\phi & \mapsto & t_{\varpi^\vee_j} \phi 
\end{eqnarray}
where $j$ is such that $\phi(-\alpha_0) = \alpha_j$,
and $j=0$ if $\phi(-\alpha_0) = -\alpha_0$.
\end{prop}
\begin{proof}
If $\phi(-\alpha_0) = -\alpha_0$ then $\phi \in \Aut(D)$ so that $\theta(\phi) = \phi = t_{\varpi^\vee_0} \phi$, since $\theta$ is the identity of $\Aut(D)$. Otherwise,  $\phi(-\alpha_0) = \alpha_j$ and by Corollary \ref{corol1:euler} we have that $v_j(-\alpha_0) = \alpha_j = \phi(-\alpha_0)$. Thus $v_j^{-1} \phi(-\alpha_0) = -\alpha_0$, so that $v_j^{-1} \phi \in \Aut(D)$. Then
$$
\theta(\phi) = \theta(v_j v_j^{-1} \phi) = \theta(v_j) \theta(v_j^{-1} \phi) =
\omega_j v_j^{-1} \phi =  
t_{\varpi^\vee_j} \phi
$$
since $\omega_j = t_{\varpi^\vee_j} v_j$, which proves our claim.
\end{proof}

Since by inspection of Table \ref{table:thm1},
$\Aut(\widetilde{D})$ is an abstract Coxeter group with generators given in the third and fourth columns of the table, 
it follows from Theorem \ref{theo:inversa-A} that $\Aut(\A)$ is abstract Coxeter and we can use Proposition \ref{prop:inversa-A-explicit}  to obtain explicit generators given by affine involutions in $V$.

Alternatively, {we can write generators of $\widetilde{W}$ given in Table \ref{table:omega} as products of Coxeter generators of $\Aut(\widetilde{D})$ given in Table \ref{table:thm1}, obtaining the following table.

\begin{figure}[h!]
\includegraphics[width=15.5cm]{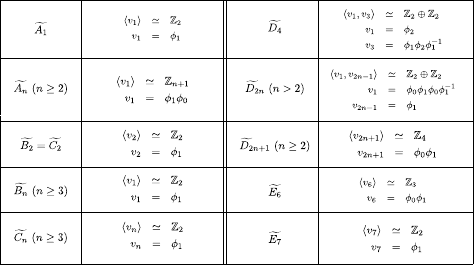}
\captionof{table}{Generators of $\widetilde{W}$ as product of generators of $\Aut(\widetilde{D})$}
\label{table:thm2}
\end{figure}

\noindent Then use this to decompose the generators of $\Aut(\widetilde{D})$ according to $\widetilde{W} \rtimes \mathrm{Aut}(D)$ and use the inverse \eqref{eq:inversa-A} to write them according to $\Omega \rtimes \Aut(D)$ in $\Aut(\A)$.}

For example, for type $\widetilde{A_n}$ $(n\geq2)$, $\Aut(\widetilde{D})$ is abstract Coxeter with generators $\phi_0, \phi_1$, where $\phi_0 \in \Aut(D)$. Their image in $\Aut(\A)$ are  
$$\tau_0 := \theta(\phi_0) = \phi_0$$
$$
\tau_1 := \theta(\phi_1) = t_{\varpi^\vee_1} \phi_1 = t_{\varpi^\vee_1} v_1 v_1^{-1} \phi_1 = 
\omega_1 \phi_0
$$
since, according to Table \ref{table:thm1}, we have $\phi_1(-\alpha_0) = \alpha_1$ and $v_1 = \phi_1 \phi_0$, so that $v_1^{-1} = \phi_0 \phi_1$. Alternatively, one can use directly Table \ref{table:thm2} to get $v_1 = \phi_1 \phi_0$, so that $\phi_1 = v_1 \phi_0$, which maps to $\tau_1 = \omega_1 \phi_0$. 
Proceeding in the same manner for the other types we get the following.

\begin{prop}
$\Aut(\A)$ is abstract Coxeter, generated by the affine involutions given in the last column of Table \ref{table:thm1}.
\end{prop}
\begin{proof}
Using Table \ref{table:thm2}, we have:
\begin{itemize}
\item For type $\widetilde{A_1}$, $\Aut(\widetilde{D})$ is abstract Coxeter with generator $\phi_1 = v_1$, so that 
$\tau_1 := \theta(\phi_1) = \omega_1$.

\item For type $\widetilde{B_2}$, $\Aut(\widetilde{D})$ is abstract Coxeter with generator $\phi_1= v_2$, so that 
$\tau_1 := \theta(\phi_1) = \omega_2$.

\item For type $\widetilde{B_n}$ $(n\geq3)$, $\Aut(\widetilde{D})$ is abstract Coxeter with generator $\phi_1= v_1$, so that $\tau_1 := \theta(\phi_1) = \omega_1$.

\item For type $\widetilde{C_n}$ $(n\geq3)$, $\Aut(\widetilde{D})$ is abstract Coxeter with generator $\phi_1= v_n$, so that $\tau_1 := \theta(\phi_1) = \omega_n$.

\item For type $\widetilde{D_4}$, $\Aut(\widetilde{D})$ is abstract Coxeter with generators
$\phi_0, \phi_1 \in \Aut(D)$ and $\phi_2 = v_1$, so that $\tau_0 := \theta(\phi_0) = \phi_0$, $\tau_1 := \theta(\phi_1) = \phi_1$, $\tau_2 := \theta(\phi_2) = \omega_1$.

\item For type $\widetilde{D_{\phantom{2}}}_{\!\!\!2n}$ $(n > 2)$, $\Aut(\widetilde{D})$  is abstract Coxeter with generators $\phi_0 \in \Aut(D)$ and $\phi_1 = v_{2n-1}$, so that $\tau_0 := \theta(\phi_0) = \phi_0$, $\tau_1 := \theta(\phi_1) = \omega_{2n-1}$.

\item For type $\widetilde{D_{\phantom{2}}}_{\!\!\!2n+1}$ $(n\geq 2)$, $\Aut(\widetilde{D})$  is abstract Coxeter with generators $\phi_0 \in \Aut(D)$ and $\phi_1 = \phi_0 v_{2n+1}$, so that $\tau_0 := \theta(\phi_0) = \phi_0$, $\tau_1 := \theta(\phi_1) = \phi_0\omega_{2n+1}$.

\item For type $\widetilde{E_6}$, $\Aut(\widetilde{D})$  is abstract Coxeter with generators $\phi_0 \in \Aut(D)$ and $\phi_1 = \phi_0 v_6$, so that $\tau_0 := \theta(\phi_0) = \phi_0$, $\tau_1 := \theta(\phi_1) = \phi_0\omega_6$.

\item For type $\widetilde{E_7}$, $\Aut(\widetilde{D})$  is abstract Coxeter with generator $\phi_1 = v_7$, so that $\tau_1 := \theta(\phi_1) = \omega_7$.
\end{itemize}
\end{proof}

Collecting the results of this section we get Theorem A of the introduction.

%%%%%%%%%%%%%%%%%%%%%%
\section{Fundamental polytope}

We start with a general result on the bounding hyperplanes 
% faces?
of the Komrakov--Premet polytope $\mathcal{K}$ \eqref{eq:kp-domain}, i.e.\ the hyperplanes $H\le V$ such that $\mathcal{K}$ entirely lies in exactly one of the closed half-spaces $H^+,H^-$ and for which $\codim(\mathcal{K}\cap H)=1$.
For this, we use that $\mathcal{K}$  has vertices (see 
Proposition 2.2.3 of \cite{Ga24})
\begin{eqnarray}
\label{eq:kp-domain-vert}
\vertices(\mathcal{K})
= \left\{\frac{\varpi^\vee_i}{m_i}\right\}_{i\in I\setminus J}
\cup\bigg\{\frac{1}{|L|+1}\sum_{j\in L}\varpi^\vee_j\bigg\}_{L\subseteq J}
\end{eqnarray}
where $J = \{ j \in I: \, m_j = 1 \}$ and $L = \emptyset$ on the right hand side furnishes the vertex $0$.
Note that
\begin{equation}
\label{eq:kp-domain-vert-alt}
\left( \alpha_0, \,\frac{\varpi^\vee_i}{m_i} \right) = 1 \quad\text{while}\quad
\left( \alpha_0, \,\frac{1}{|L|+1}\sum_{j\in L}\varpi^\vee_j \right) = \frac{|L|}{|L|+1}
< 1
\end{equation}
Consider the hyperplanes
\begin{align*}
&H_i:=\{x \in V~|~(\alpha_i, x)=0\} \quad (i \in I),
\qquad H_0:=\{x \in V~|~(\alpha_0, x)=1\},\\
&H_j^0:=\{x \in V~|~(\alpha_0+\alpha_j, x)=1\}\quad (j \in J)
\end{align*}
illustrated below for type $A_2$.

\begin{center}
\def\svgwidth{10cm}
%% Creator: Inkscape 1.0.1 (c497b03c, 2020-09-10), www.inkscape.org
%% Accompanies image file 'A2-vertex-top-face.pdf' (pdf, eps, ps)
\begingroup%
  \makeatletter%
  \providecommand\color[2][]{%
    \errmessage{(Inkscape) Color is used for the text in Inkscape, but the package 'color.sty' is not loaded}%
    \renewcommand\color[2][]{}%
  }%
  \providecommand\transparent[1]{%
    \errmessage{(Inkscape) Transparency is used (non-zero) for the text in Inkscape, but the package 'transparent.sty' is not loaded}%
    \renewcommand\transparent[1]{}%
  }%
  \providecommand\rotatebox[2]{#2}%
  \newcommand*\fsize{\dimexpr\f@size pt\relax}%
  \newcommand*\lineheight[1]{\fontsize{\fsize}{#1\fsize}\selectfont}%
  \ifx\svgwidth\undefined%
    \setlength{\unitlength}{91.20580335bp}%
    \ifx\svgscale\undefined%
      \relax%
    \else%
      \setlength{\unitlength}{\unitlength * \real{\svgscale}}%
    \fi%
  \else%
    \setlength{\unitlength}{\svgwidth}%
  \fi%
  \global\let\svgwidth\undefined%
  \global\let\svgscale\undefined%
  \makeatother%
  \begin{picture}(1,0.51434613)%
    \lineheight{1}%
    \setlength\tabcolsep{0pt}%
    \put(0,0){\includegraphics[width=\unitlength,page=1]{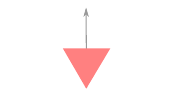}}%
    \put(0.83863704,0.27319011){\makebox(0,0)[lt]{\lineheight{1.25}\smash{\begin{tabular}[t]{l}$\alpha_1$\end{tabular}}}}%
    \put(0.44556509,0.48863277){\color[rgb]{0.6,0.6,0.6}\makebox(0,0)[lt]{\lineheight{1.25}\smash{\begin{tabular}[t]{l}$\alpha_0$\end{tabular}}}}%
    \put(0,0){\includegraphics[width=\unitlength,page=2]{A2-vertex-top-face.pdf}}%
    \put(0.17778184,0.27569064){\color[rgb]{0.83137255,0,0}\makebox(0,0)[lt]{\lineheight{1.25}\smash{\begin{tabular}[t]{l}$H_0$\end{tabular}}}}%
    \put(0.2789495,0.37525751){\color[rgb]{0.83137255,0,0}\makebox(0,0)[lt]{\lineheight{1.25}\smash{\begin{tabular}[t]{l}$H_1$\end{tabular}}}}%
    \put(0,0){\includegraphics[width=\unitlength,page=3]{A2-vertex-top-face.pdf}}%
    \put(0.02951428,0.27319011){\makebox(0,0)[lt]{\lineheight{1.25}\smash{\begin{tabular}[t]{l}$\alpha_2$\end{tabular}}}}%
    \put(0.2764601,0.06441855){\color[rgb]{0,0,1}\makebox(0,0)[lt]{\lineheight{1.25}\smash{\begin{tabular}[t]{l}$H^0_2$\end{tabular}}}}%
    \put(0.60538638,0.06441855){\color[rgb]{0,0,1}\makebox(0,0)[lt]{\lineheight{1.25}\smash{\begin{tabular}[t]{l}$H^0_1$\end{tabular}}}}%
    \put(0.58926063,0.37525751){\color[rgb]{0.83137255,0,0}\makebox(0,0)[lt]{\lineheight{1.25}\smash{\begin{tabular}[t]{l}$H_2$\end{tabular}}}}%
    \put(0.46274222,0.13007965){\makebox(0,0)[lt]{\lineheight{1.25}\smash{\begin{tabular}[t]{l}$\mathcal{L}$\end{tabular}}}}%
    \put(0,0){\includegraphics[width=\unitlength,page=4]{A2-vertex-top-face.pdf}}%
    \put(0.43849775,0.21568542){\makebox(0,0)[lt]{\lineheight{1.25}\smash{\begin{tabular}[t]{l}$\mathcal{A}$\end{tabular}}}}%
    \put(0.41205471,0.13007965){\color[rgb]{1,1,1}\makebox(0,0)[lt]{\lineheight{1.25}\smash{\begin{tabular}[t]{l}$\mathcal{K}$\end{tabular}}}}%
    \put(0.46841933,0.30858332){\color[rgb]{0,0,0}\makebox(0,0)[lt]{\lineheight{1.25}\smash{\begin{tabular}[t]{l}$H_0^0$\end{tabular}}}}%
  \end{picture}%
\endgroup%

\end{center}

%It is known that, in contrast with $\Omega$, the full isometry group $\Aut(\A)$ of the alcove $\A$ is a Coxeter group of dihedral type (except for the triality case $D_4$, where it is $\Sym_4$). While this group does \emph{not} act as a reflection group on $V^*$ in general, we shall see that we still can find a fundamental polytope $\mathcal{L}$ for its action inside the Komrakov--Premet polytope $\K$.
%
%To do this, we identity $\Aut(\A)$ with the automorphism group of the extended Dynkin diagram $\widetilde{\mathcal{D}}$ of $\Phi$, that splits as a semi-direct product
%\[\Aut(\A)\simeq\Aut(\widetilde{\mathcal{D}})\simeq \Omega\rtimes \Aut(D),\]
%where $\mathcal{D}$ is the usual (finite) Dynkin diagram of $\Phi$.
%
%Of course, when $\Aut(D)=1$, we have $\Aut(\A)=\Omega$ and we can take $\mathcal{L}=F$, so there is nothing to do. Thus, we shall concentrate on types $A_{n\ge2}$, $D_{n\ge4}$ and $E_6$, which we will treat separately, but following the same strategy and lines. The main idea is to circumvent the fact that the isometric involutions generating $\Aut(D)$ are not reflections in $V^*$, by defining vectors in $V$ that play the role of ``roots'' for these involutions and take there orthogonal complements in $V^*$. The resulting hyperplanes replace the fixed points of the involutions (which are not hyperplanes) and allow to slice the polytope $\K$ to obtain $\mathcal{L}$. 

\begin{lem}\label{lem:boundings_F}
If $\Omega\ne \{1\}$, the bounding hyperplanes for $\mathcal{K}$ are precisely the $H_i$'s, for $i\in I$, together with the $H_j^0$'s for $j\in J$. In particular, $H_0$ does not bound $\K$.
\end{lem}
\begin{proof}
From \eqref{eq:kp-domain} it follows that $\mathcal{K}=H_0^-\cap\left(\bigcap_{i\in I}H_i^+\right)\cap\left(\bigcap_{j\in J}(H_j^0)^-\right)$. For $i\in I$, the vertices $0$ and $\varpi^\vee_k/m_k$ ($i\ne k\in I\setminus J$) and $\varpi^\vee_k/2$ ($i\ne k\in J$) belong to $H_i$, so that $\dim(\mathcal{K}\cap H_i)=n-1$.
On the other hand, for $j\in J\ne\emptyset$, the vertices $\varpi^\vee_k/m_k$ ($k\in I\setminus J$) and $\frac{1}{|J'|+1}\sum_{l\in J'}\varpi^\vee_l$ ($j\in J'\subseteq J$) belong to $H_j^0$, and there are $n-|J|+2^{|J|-1}\ge n$ of them, so that again $\dim(F\cap H_j^0)=n-1$. Finally, the only vertices of $\mathcal{K}$ belonging to $H_0$ are the $\varpi^\vee_k/m_k$ for $k\in I\setminus J$, which are in number $|I\setminus J|\le n-1$ and thus $\dim(\mathcal{K}\cap H_0)<n-1$.
\end{proof}

We now consider the cases with $\Aut(D) \neq 1$. Let $\phi_0 \in \Aut(D)$ be a non-trivial involution. 

\begin{prop}
\label{prop:u0-slice}
If $0 \not=v_0 \in V$ satisfies $\phi_0(v) = -v_0$, then slicing $\mathcal{K}$ with
\begin{equation}
H_0^0:=\{x \in V~|~(v_0, x)=0\},
\end{equation}
we get the fundamental polytope
\begin{equation}
\mathcal{L}:=\{x \in \mathcal{K}~|~(v_0, x)\ge0\}
\end{equation}
for $\left<\phi_0\right>$ acting on $\mathcal{K}$.
\end{prop}
\begin{proof}
Since $\phi_0$ is an isometric involution such that $\phi_0(v_0)=-v_0$, we have
$$
(\phi_0(x), v_0) = (x, \phi_0(v_0))=-(x, v_0)
$$
It follows that, if $x\in \mathcal{K}$ does not belong to $\mathcal{L}$, then $\phi_0(x)$ does. Also, the intersection $\mathcal{L}\cap \phi_0(\mathcal{L})$ is included in the hyperplane $\{v_0 = 0\}$, and so has empty interior. This proves that $\mathcal{L}$ is a fundamental domain for $\left<\phi_0\right>$ acting on $\K$. Furthermore, it is clear that $\mathcal{L}$ is a polytope, being the intersection of the polytope $\K$ with the closed half-space $\{ v_0 \geq 0\}$.
\end{proof}

To obtain $v_0$ as in the preceding proposition, let $\phi_0(\alpha_i) = \alpha_{\pi(i)}$, then the permutation $\pi \neq 1$ is an involution of $I$ (for later use note that, by Proposition \ref{lemma1}, $\pi$ leaves $J$ invariant).
%\red{KH: What does non-trivial  mean  in this context? That there are no fixed points?}
Choose a non-empty $J_0\subseteq I$ which decomposes as a disjoint union $J_0 = J^+_0 \cup J^-_0$ such that $\pi(J^+_0) = J^-_0$ and introduce the corresponding {\it balanced root with support $J_0$}
\begin{eqnarray}
\label{eq:def-balanced}
v_0 & := & \sum_{i \in J^+_0} \alpha_i - \sum_{i \in J^-_0} \alpha_i  = \sum_{i \in J^+_0} \alpha_i - \alpha_{\pi(i)}
\end{eqnarray}
Clearly, $v_0 \neq 0$ and $\phi_0(v_0) = -v_0$, which is enough to get the fundamental domain $\mathcal{L}$ of Proposition \ref{prop:u0-slice}.
For $L \subseteq I$ we have the inner product
\begin{equation}
\label{eq:inner-product-vertex}
\Big(v_0, \, \sum_{j \in L} \varpi^\vee_j \Big) = |L \cap J^+_0| - |L \cap J^-_0|
\end{equation}

From now on we will consider a fundamental domain $\mathcal{L}$ obtained from $v_0$ with the additional condition that its support
$J_0$ is contained in $J$, the so called balanced minuscule coroots.  
%We will show that a fundamental domain $\mathcal{L}$ built with such a $v_0$ is a subpolytope of $\K$, that is, the vertices of $\mathcal{L}$ are a subset of the vertices of $\mathcal{K}$.  
%First, 
%Some preparatory remarks are in order.
A nonempty face $\mathcal{F} \subseteq \mathcal{L}$ can be written (non-uniquely) as an intersection
\begin{equation}
\label{eq:def-face}
\mathcal{F}
%= \mathcal{F}(\mathcal{L}, K, L)
=\mathcal{L}\cap\bigg(\bigcap_{i\in K}H_i\cap\bigcap_{j\in L}H_j^0\bigg)
\end{equation}
where $K \subseteq I$, $L \subseteq J \cup \{0\}$. In particular, $x \in \mathcal{F}$ has the form
\begin{equation}
\label{eq:x-generic}
x = \sum_{i \in I \setminus (K \cup L)} (\alpha_i, x)\varpi^\vee_i
+ 
(1 - (\alpha_0,x))\sum_{j \in L \cap I} \varpi^\vee_j
\end{equation}
Clearly, there exists $K, L$ with minimal cardinality such that $\codim \mathcal{F} = |K| + |L|$ (the number of independent equations that define the face).  
But note that $K$ and $L$ can intersect: 
for example, a vertex $\varpi^\vee_i/m_i$ of $\K$ in \eqref{eq:kp-domain-vert}, $i\in I\setminus J$, belongs to any $H_j \cap H_j^0$, $j \in J \setminus\{i\}$, as in the example below in type $B_2$.  
%$j_0 \in K \cap L$ implies
%$\mathcal{F} \subseteq H_{j_0} \cap H^0_{j_0} \subseteq H_0$, so that in this case $\mathcal{K}$ meets the upper face $H_0$ of the fundamental alcove $\mathcal{A}$ in a root hyperplane, as in the example below in type $B_2$.
%%In particular, if $\mathcal{K}$ does not meet $H_0$, then $K \cap L = \emptyset$.
%% which happens in type $A$. 
\begin{center}
\def\svgwidth{13cm}
%% Creator: Inkscape 1.0.1 (c497b03c, 2020-09-10), www.inkscape.org
%% Accompanies image file 'B2-vertex-top-face.pdf' (pdf, eps, ps)
\begingroup%
  \makeatletter%
  \providecommand\color[2][]{%
    \errmessage{(Inkscape) Color is used for the text in Inkscape, but the package 'color.sty' is not loaded}%
    \renewcommand\color[2][]{}%
  }%
  \providecommand\transparent[1]{%
    \errmessage{(Inkscape) Transparency is used (non-zero) for the text in Inkscape, but the package 'transparent.sty' is not loaded}%
    \renewcommand\transparent[1]{}%
  }%
  \providecommand\rotatebox[2]{#2}%
  \newcommand*\fsize{\dimexpr\f@size pt\relax}%
  \newcommand*\lineheight[1]{\fontsize{\fsize}{#1\fsize}\selectfont}%
  \ifx\svgwidth\undefined%
    \setlength{\unitlength}{133.75295874bp}%
    \ifx\svgscale\undefined%
      \relax%
    \else%
      \setlength{\unitlength}{\unitlength * \real{\svgscale}}%
    \fi%
  \else%
    \setlength{\unitlength}{\svgwidth}%
  \fi%
  \global\let\svgwidth\undefined%
  \global\let\svgscale\undefined%
  \makeatother%
  \begin{picture}(1,0.34634876)%
    \lineheight{1}%
    \setlength\tabcolsep{0pt}%
    \put(0,0){\includegraphics[width=\unitlength,page=1]{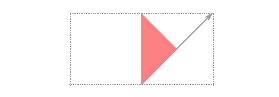}}%
    \put(0.77993298,0.04212782){\makebox(0,0)[lt]{\lineheight{1.25}\smash{\begin{tabular}[t]{l}$\alpha_1$\end{tabular}}}}%
    \put(0.20798303,0.31128085){\makebox(0,0)[lt]{\lineheight{1.25}\smash{\begin{tabular}[t]{l}$\alpha_2$\end{tabular}}}}%
    \put(0,0){\includegraphics[width=\unitlength,page=2]{B2-vertex-top-face.pdf}}%
    \put(0.53183744,0.11820006){\color[rgb]{1,1,1}\makebox(0,0)[lt]{\lineheight{1.25}\smash{\begin{tabular}[t]{l}$\mathcal{K}$\end{tabular}}}}%
    \put(0.53183744,0.19670299){\makebox(0,0)[lt]{\lineheight{1.25}\smash{\begin{tabular}[t]{l}$\mathcal{A}$\end{tabular}}}}%
    \put(0,0){\includegraphics[width=\unitlength,page=3]{B2-vertex-top-face.pdf}}%
    \put(0.44211963,0.13990414){\color[rgb]{0,0,1}\makebox(0,0)[lt]{\lineheight{1.25}\smash{\begin{tabular}[t]{l}$H_2^0$\end{tabular}}}}%
    \put(0.72248696,0.23360157){\color[rgb]{0.66666667,0,0}\makebox(0,0)[lt]{\lineheight{1.25}\smash{\begin{tabular}[t]{l}$H_2$\end{tabular}}}}%
    \put(0,0){\includegraphics[width=\unitlength,page=4]{B2-vertex-top-face.pdf}}%
    \put(0.72248696,0.10101507){\color[rgb]{0.66666667,0,0}\makebox(0,0)[lt]{\lineheight{1.25}\smash{\begin{tabular}[t]{l}$H_0$\end{tabular}}}}%
    \put(0,0){\includegraphics[width=\unitlength,page=5]{B2-vertex-top-face.pdf}}%
    \put(0.77993298,0.31128085){\color[rgb]{0.6,0.6,0.6}\makebox(0,0)[lt]{\lineheight{1.25}\smash{\begin{tabular}[t]{l}$\alpha_0$\end{tabular}}}}%
  \end{picture}%
\endgroup%

\end{center}

Write $x \in \vertices(\mathcal{L})$ as the face
\begin{equation}
\label{eq:def-vert}
\{x\}=\mathcal{L}\cap\Big(\bigcap_{i\in K}H_i\cap\bigcap_{j\in L}H_j^0\Big)
\end{equation}
with codimension $|K| + |L| = n$. If $0 \not\in L$ then it is immediate that $x$ is a vertex of $\mathcal{K}$, since the slicing hyperplane $H_0^0$ does not appear in the above intersection. If $0 \in L$ then $x$ does belong to $H_0^0$. If, furthermore, if $x$ is not a vertex of $\K$, we show below that $K \cap L = \emptyset$ because, loosely speaking, the balanced minuscule root makes $x$ sit in the middle of $\K$.

\begin{lem}
Let $v_0$ be a balanced minuscule root.
\label{lem:empty}
If $0 \in L$ and $|K| + |L| = n$, then either $K \cap L = \emptyset$ or $x\in\vertices(\mathcal{K})$.
\end{lem}
\begin{proof}
To the intersections \eqref{eq:def-vert} corresponds the minimum number of $|K| + |L| = n$ equations 
\begin{equation}
\label{eq:x}
(\alpha_i, x) = 0 \quad (i \in K) \qquad 
(\alpha_0 + \alpha_j, x) = 1 \quad (j \in L \cap I) \qquad
(v_0,x) = 0
\end{equation}
that determine $x$ uniquely.
Suppose there exists $j_0 \in K \cap L$, then $(\alpha_0 + \alpha_{j_0}, x) = (\alpha_0, x) = 1$. 
Then, for every $j \in L \cap I$,
the equation $(\alpha_0 + \alpha_j, x) = 1$ implies $\alpha_j(x) = 0$,
so that all these equations for $x$ become the same equation
$(\alpha_0, x) = 1$. By the minimality of $n$, it follows that $L \cap I = \{ j_0 \}$, so that $L = \{0, j_0\}$ and $|L|=2$. Thus $|K| = n-2$,  so that $I \setminus K = \{r, s\}$ and thus
\[ x = a_r \varpi^\vee_r + a_s \varpi^\vee_s \in H_0^0 \cap H_0 \quad \mbox{ for } \quad
  a_r, a_s \geq 0.\] Equation $(\alpha_0, x) = 1$ imposes
\begin{equation}
\label{eq:x0}
m_r a_r + m_s a_s = 1
\end{equation}
In order that the equation $(v_0, x) = 0$ for $H_0^0$ also
imposes a non-trivial equation for $a_r$,$a_s$, it is necessary that $J_0 \cap \{r, s\} \neq \emptyset$, otherwise it would yield the trivial equation $0=0$ and \eqref{eq:x} would not determine $x$ uniquely. 

We claim that $|J_0 \cap \{r,s\}| = 2$ is impossible. Indeed, 
if $J_0 \cap \{r,s\} = \{r,s\}$ then
%$m_s = m_r = 1$, so that 
$(v_0, x) = 0$ imposes two possibilities. 
Either $m_r a_r + m_s a_s = 0$, contradicting \eqref{eq:x0}.
Or $m_r a_r - m_s a_s = 0$, which together with \eqref{eq:x0} implies $m_r a_r = m_s a_s = \frac{1}{2}$. It follows that
$$
(\alpha_0 + \alpha_r, x) = m_r a_r + m_s a_s + a_r = 1 + \frac{1}{2m_r} > 1
$$
contradicting $x \in \K$, since $r \in J$. 

Finally, we claim that  $|J_0 \cap \{r,s\}| = 1$ implies $x \in \vertices(\K)$.
Indeed, if $J_0 \cap \{r,s\} = \{r\}$, then $(v_0, x) = 0$ imposes  $a_r = 0$, so that $a_s = 1/m_s$ and $x = \varpi^\vee_s/m_s$. If $s \in J$, then $m_s=1$ and $(\alpha_0+\alpha_s,x)=2>1$, contradicting $x\in\mathcal{K}$. If $s \not\in  J$, then $x\in\vertices(\mathcal{K})$ by \eqref{eq:kp-domain-vert}.
The case $J_0\cap\{r,s\}=\{s\}$ is analogous.
\end{proof}

\begin{theo}
\label{theo:vertex}
Let $v_0$ be a balanced minuscule root.
We have that
\begin{eqnarray}
\vertices(\mathcal{L}) &=& \mathcal{L} \cap \vertices(\mathcal{K}) \nonumber\\
\label{A-vert}
&=& \Big\{\frac{\varpi^\vee_i}{m_i}\Big\}_{i\in I\setminus J}
\cup\bigg\{\frac{1}{|L|+1}\sum_{j\in L}\varpi^\vee_j
\mid L\subseteq J, \quad |L \cap J^+_0| \geq |L \cap J^-_0|
\bigg\} \nonumber
\end{eqnarray}
\end{theo}
\begin{proof}
Since $\mathcal{L} \subseteq \K$, we clearly have $\mathcal{L} \cap \vertices(\K) \subseteq \vertices(\mathcal{L})$.  For the reciprocal inclusion, let $x \in \vertices(\mathcal{L})$. By Lemma \ref{lem:empty} and the discussion preceding it, we can write $x$ as in  \eqref{eq:def-vert} with $|K|+|L| = n$, and assume that $0 \in L$, $K \cap L = \emptyset$.
It follows that 
\[ 
\label{eq:edge}
x \in 
%\mathcal{F}(\mathcal{K}, K, L \cap I) 
%=
\mathcal{F} := \mathcal{K}\cap\bigg(\bigcap_{i\in K}H_i\cap\bigcap_{j\in L \cap I}H_j^0\bigg),
\]
%so that $\mathcal{F}(\mathcal{K}, K, L \cap I)$ 
which is an edge of $\K$, since $|K| + |L \cap I| = n-1$. 
Thus
\begin{equation}
\label{eq:comb-convex}
x=(1-t)x_0 + tx_1
\end{equation}
for $x_0, x_1 \in \vertices(\K) \cap \mathcal{F}$ and $t\in[0,1]$.
We can assume that $t \in (0,1)$, otherwise we already get that $x \in \vertices(\K)$. Furthermore, we can assume that both $x_0, x_1 \neq 0$. Indeed, if $x_0 = 0$ then $0=(v_0,x) = t(v_0,x_1)$ so that both $(v_0,x_0) = (v_0,x_1)=0$, since $t>0$. It follows that $x_0,x_1 \in H_0^0 \cap \mathcal{K}\subseteq\mathcal{L}$ and, since $x$ is a vertex of $\mathcal{L}$, it follows that $x=x_0=x_1$ and we already get that $x \in \vertices(\K)$.
By \eqref{eq:x-generic} and the disjoint union
\begin{equation}
\label{eq:disj-union-faces}
I = K \cup (L \cap I) \cup \{j_0\}
\end{equation}
$x, x_0, x_1 \in \mathcal{F}$ have the form
\begin{equation}
\label{eq:x-genericj0}
z = (\alpha_{j_0}, z)\varpi^\vee_{j_0} + (1 - (\alpha_0,z))\sum_{j \in L \cap I} \varpi^\vee_j, 
\end{equation}
Denote $\ell:=|L\cap I|+1$. 
Since $x_0 \in \vertices(\K)$, either
\begin{enumerate}[label=\roman*)]
\item There is an additional $H_{i_0}$ with $i_0 \in I \setminus K$,
  in this case either:
\begin{align*}
x_0=\frac{\varpi^\vee_{j_0}}{m_{j_0}}\quad(j_0 \not\in J)
\qquad\text{or}\qquad
x_0=\frac1\ell\sum_{j\in L\cap I}\varpi^\vee_j &&
\end{align*}
Indeed, by \eqref{eq:disj-union-faces} either $i_0 \in L$ or $i_0 = j_0$.
Since $\alpha_{i_0}(x_0) = 0$, if $i_0\in L$, then 
$$
1 = (\alpha_0 + \alpha_{i_0}, x_0) = (\alpha_0,x_0)
$$
so that by \eqref{eq:x-genericj0} and the description \eqref{eq:kp-domain-vert}, \eqref{eq:kp-domain-vert-alt} of the vertices of $\mathcal{K}$, we get
$x_0=\varpi^\vee_{j_0}/m_{j_0}$, with $j_0 \not\in J$. Otherwise, $i_0=j_0$ so that by \eqref{eq:x-genericj0} we have $x_0 = (1 - (\alpha_0,x))\sum_{j \in L \cap I} \omega_j$.
Since $L \cap I \subseteq J$, by the description \eqref{eq:kp-domain-vert}, \eqref{eq:kp-domain-vert-alt} of the vertices of $\mathcal{K}$, we get $x_0=\frac1\ell\sum_{j\in L\cap I}\varpi^\vee_j$, as claimed.

\item Or there is an additional $H_{i_0}^0$ with $i_0 \in J \setminus L$, in this case either:
\begin{align*}
x_0=\frac{\varpi^\vee_{j_0}}{m_{j_0}}\quad(j_0 \not\in J)
\qquad\text{or}\qquad
x_0=\frac{1}{\ell+1}\left(\varpi^\vee_{j_0}+\sum_{j\in L\cap I}\varpi^\vee_j\right)
\quad(i_0 = j_0 \in J) &&
\end{align*}
Indeed, by \eqref{eq:disj-union-faces} either $i_0 \in K \cap J$ or $i_0 = j_0 \in J$. By \eqref{eq:x-genericj0}, if $i_0\in K$, then arguing as in the previous case, we get that $x_0=\varpi^\vee_{j_0}/{m_{j_0}}$ for $j_0 \not\in J_0$.
Otherwise, we have $i_0=j_0 \in J$, so that
$x_0\in\bigcap_{j\in(L\cap I)\cup\{j_0\}}H_j^0$, thus
$x_0=\left(\varpi^\vee_{j_0}+\sum_{j\in L\cap I}\varpi^\vee_j\right)/(\ell + 1)$,
as claimed.
\end{enumerate}
Respectively for $x_1 \in \vertices(\K)$, there is an additional $H_{i_1}$ with $i_1 \in I \setminus K$ or $H_{i_1}^0$ with $i_1 \in J \setminus L$ containing $x_1$.
%%%%%%

We claim that $x_0 = x_1$, so that $x=x_0=x_1$ and $x \in \vertices(\K)$, as claimed. Indeed, we suppose that $x_0 \neq x_1$ and distinguish four cases.
\begin{enumerate}[label=$\bullet$]

\item Suppose that $x_0\in H_{i_0}^0$ and $x_1\in H_{i_1}^0$, 
$i_0,i_1\in J \setminus L$. 
By items (i) and (ii) above this cannot occur because either $j_0\in J$ or not, but not both.

\item Suppose that $x_0\in H_{i_0}$ and $x_1\in H_{i_1}$, so that $i_0,i_1 \in I \setminus K$. 
By items (i) and (ii) above, we can assume that
$x_0=\frac{\varpi^\vee_{j_0}}{m_{j_0}}$, $(j_0 \not\in J)$, $x_1=\frac1\ell\sum_{j\in L\cap I}\varpi^\vee_j$. Since $j_0$ is not in the support of $v_0$, we get
$(v_0,x_0) = 0$ and 
\[
0=(v_0,x)=t(v_0,x_1)
\]
so that also $(v_0,x_1)=0$, since $t>0$. It follows that $x_0,x_1 \in H_0^0 \cap \mathcal{K}\subseteq\mathcal{L}$ and, since $x$ is a vertex of $\mathcal{L}$, it follows that $x=x_0=x_1$, a contradiction.

\item Suppose that $x_0\in H_{i_0}$ with $i_0 \in I \setminus K$ and 
$x_1\in H_{i_1}^0$ with $i_1\in J \setminus L$. 
In item (i) above we cannot have $x_0=\varpi^\vee_{j_0}/m_{j_0}$, $j_0 \not\in J$, since by item (ii) the only possibility for $x_1 \neq x_0$ would have $j_0 \in J$, a contradiction.
Thus, $x_0 =(\sum_{j\in L\cap I}\varpi^\vee_j) / \ell$. By item (ii) we have two possibilities for $x_1$.

If $x_1 = \varpi^\vee_{j_0}/m_{j_0}$, $(j_0 \not\in J)$, then $j_0$ is not in the support of $v_0$. Thus we get $(v_0,x_1)=0$ and
$$
0 = (v_0, x) = (1-t)(v_0,x_0)
$$
so that both $(v_0,x_0) = (v_0,x_1)=0$, since $t<1$. It follows that $x_0,x_1 \in H_0^0 \cap \mathcal{K}\subseteq\mathcal{L}$ and, since $x$ is a vertex of $\mathcal{L}$, it follows that $x=x_0=x_1$, a contradiction.

If $x_1=\left(\varpi^\vee_{j_0}+\sum_{j\in L\cap I}\varpi^\vee_j\right)/(\ell +1)$
then
\[x=\frac{t}{\ell+1}\varpi^\vee_{j_0}+\left(\frac{t}{\ell+1}+\frac{1-t}{\ell}\right)\sum_{j\in L\cap I}\varpi^\vee_j.\]
We can thus compute 
\[
0=
(v_0, x)=\left(\frac{t}{\ell+1}+\frac{1-t}{\ell}\right)
(|L\cap J^+_0|-|L\cap J^-_0|)
+
\left(
\frac{t}{\ell+1}
\right)
\cdot
\left\{
\begin{array}{cc}
1 & \text{if $j_0 \in J^+_0$} \\[.5em]
0 & \text{if $j_0 \not\in J_0$} \\[.5em]
-1 & \text{if $j_0 \in J^-_0$}
\end{array}\right.
\]
If $j_0 \in J^-_0$ then the above equality implies that $|L\cap J^+_0|-|L\cap J^-_0| > 0$ so that
\[
\frac{t}{\ell+1}+\frac{1-t}{\ell}\le \left(\frac{t}{\ell+1}+\frac{1-t}{\ell}\right)
(|L\cap J^+_0|-|L\cap J^-_0|)
=\frac{t}{\ell+1}
\]
thus $1-t\le 0$, i.e. $t=1$, which contradicts our assumption that $t \in (0,1)$. 
The case $j_0 \in J^+_0$ is entirely analogous.  
%If $j_0 \in J^+_0$ then $|L\cap J^+_0|-|L\cap J^-_0| < 0$ so that
%\[
%-\frac{t}{\ell+1}-\frac{1-t}{\ell} \geq \left(-\frac{t}{\ell+1}-\frac{1-t}{\ell}\right)
%(|L\cap J^-_0|-|L\cap J^+_0|)
%=-\frac{t}{\ell+1}
%\]
%thus $-(1-t) \geq 0$, i.e. $t=1$, as before.
%\orange{We can say that this is analogous to the previous case ($j_0\in J_0^-$), by symmetry.}  
%%
If $j_0 \not\in J_0$, then the above equality implies $|L \cap J^+_0| = |L \cap J^-_0|$ so that, by \eqref{eq:inner-product-vertex}, we get $(v_0,x_1) = 0$.  Then $t \in (0,1)$ implies $(v_0,x_0) = 0$ so that both $x_0, \, x_1\in H_0^0 \cap \K \subseteq \mathcal{L}$.  Since $x$ is a vertex of $\mathcal{L}$, it follows that $x = x_0 = x_1$, a contradiction.

\item The case $x_0\in H_{i_0}^0$ and $x_1\in H_{i_1}$ is treated similarly, changing $t$ into $1-t$.
\end{enumerate}
This proves the first equality $\vertices(\mathcal{L}) = \mathcal{L} \cap \vertices(\mathcal{K})$ of the theorem. The second equality then follows from \eqref{eq:kp-domain-vert}, \eqref{eq:inner-product-vertex} and the definition of $\mathcal{L}$.
%The last claim follows from
%$$
%\label{eq:edge}
%x \in \mathcal{K}\cap\left(\bigcap_{i\in K}H_i\cap\bigcap_{j\in L \cap I}H_j^0\right)
%$$
%which is an edge of $\K$, since $|K|+|L| = n$ and $0 \in L$, so that that $|K| + |L \cap I| = n-1$.
\end{proof}

In the following subsections, we specify the results of this section. Of course, when $\Aut(D)=1$, we have $\Aut(\A)=\Omega$ and we can take $\mathcal{L}=\mathcal{K}$, so there is nothing to do. Thus, it is enough to focus on types $A_{n\ge2}$, $D_{n\ge4}$ and $E_6$.

\subsection{Type $A_{n\ge2}$}
The highest root is $\alpha_0 = \alpha_1 + \alpha_2 + \ldots + \alpha_n$ so that $J = I$ and the Komrakov--Premet polytope $\K$ has vertices 
\begin{equation}
\label{eq:vert-KP}
\mathrm{vert}(\mathcal{K}) =
\left\{\frac{1}{|L|+1}\sum_{j\in L}\varpi^\vee_j\right \}_{L\subseteq I}
\end{equation}
It follows that $|\mathrm{vert}(\K)| = 2^{|I|} = 2^n$ and that $\K$ does not meets the upper face $H_0$ of $\A$.
By Table \ref{table:thm1} $\Aut(D)$ is generated by one isometric involution $\phi_0$, which acts on $\Pi$ by $\phi_0(\alpha_i)=\alpha_{n+1-i}$. It completes the fundamental group $\Omega \simeq \Z_{n+1}$ in $\Aut(\A)$ 
in such a way that
$$
\Aut(\A) \simeq I_2(n+1)
$$
%Note that, denoting by $w_0\in W$ the longest element of $W=\Sym_{n+1}$, we have that $\phi_0=-w_0\notin W$. 
%

We choose the maximal support for a minuscule balanced root $v_0$, given by $J^+_0 = \{1,\ldots, \lfloor n/2 \rfloor \}$, 
$J^-_0 = \{ \lceil n/2 \rceil + 1, \ldots, n \}$ so that
\[
v_0:=\sum_{i=1}^{\lfloor \frac{n}{2} \rfloor}
(\alpha_i-\alpha_{n+1-i})
\]
By Proposition \ref{prop:u0-slice}, $\mathcal{L}$ is a fundamental polytope for $\Aut(D)=\left<\phi_0\right>$ acting on $\mathcal{K}$. Theorem \ref{theo:vertex} then implies the following.

\begin{theo}\label{thm:fund_dom_An}
$\mathcal{L}$ is a fundamental polytope for the action of $\Aut(\A)$ in $\A$ such that
\begin{eqnarray}
\label{A-vertb}
\vertices(\mathcal{L}) &=& \vertices(\mathcal{K}) \cap \mathcal{L} \nonumber\\
&=& 
\left\{\left.\frac{1}{|L|+1}\sum_{j\in L}\varpi^\vee_j~\right|
~L\subseteq I,~\,
\left|L\cap\left\{1,\dotsc, \left\lfloor \frac{n}{2} \right\rfloor \right\}\right|
\geq
\left|L\cap\left\{\left\lceil \frac{n}{2} \right\rceil + 1,\dotsc,n\right\}\right|
\right\}
\end{eqnarray}
\end{theo}

\begin{rem}
For $n=2, 3$, the isometric involution $\phi_0$ is the orthogonal reflection in $V$ around $\{ \alpha_0 = 0 \}$.
For $n \geq 4$, the isometric involution $\phi_0$ is not a reflection.
Indeed, let $k = \left\lfloor \frac{n}{2} \right\rfloor$, the subspace 
\[
\ker(\phi_0-1)= 
\bigoplus_{i=1}^k \left<\alpha_i+\alpha_{n-i}\right>
\oplus
\left\{
\begin{array}{cr}
0& n\text{ even } \\
\left<\alpha_{ k + 1 } \right>& n\text{ odd } 
\end{array}
\right.
\]
having codimension $k>1$ in $V$, is not a hyperplane. We thus cannot use it to slice the Komrakov--Premet domain $\K$. We circumvent this technicality by rather looking at the hyperplane orthogonally completing the line 
\[\left<v_0\right>\le\ker(\phi_0+1)=\bigoplus_{i=1}^k \left<\alpha_i-\alpha_{n-i}\right>\]
%in $V$.
%In some sense, the vector $v_0$ plays the role of a ``highest weight vector'' in $\ker(\phi_0+1)$, generalizing the notion of root of a reflection.
\end{rem}

\begin{prop}\label{prop:nr_vertices}
%The domain $\K$ has $2^n$ vertices, while for $\mathcal{L}$, we have
%\[
$|\vertices(\mathcal{L})|=2^{n-1}+\frac{3-(-1)^n}{4}\binom{2\lfloor n/2\rfloor}{\lfloor n/2\rfloor}$
%\cancel{\sum_{i=0}^{\lfloor n/2\rfloor}\binom{\lfloor n/2\rfloor}{i}^2}.
%\]
\end{prop}
\begin{proof}
Denote $k:=\lfloor n/2\rfloor$. First, we find the cardinality of the set $\{\mu\in\vertices(\K)~|~\left<\mu,v_0\right>=0\}$. To an element of this set corresponds a subset of $\{1,\dotsc,k\}$, together with a subset of $\{k+(3-(-1)^n)/2,\dotsc,n\}$ of the same cardinality, maybe together with $\{(n+1)/2\}$ (when $n$ is odd). Partitioning the set of subsets of $\{1,\dotsc,k\}$ by cardinalities, we find that
\[|\{\mu\in\vertices(\K)~|~\left<\mu,v_0\right>=0\}|=\sum_{i=0}^k\binom{k}{i}^2\times\left\{\begin{array}{cc}1 & \text{if $n$ is even}, \\ 2 & \text{if $n$ is odd}.\end{array}\right.=\frac{3-(-1)^n}{2}\sum_{i=0}^k\binom{k}{i}^2.\]
Because $\phi_0$ acts on $\vertices(\K)$, we also find
\[2^n=|\vertices(\K)|=2|\{\mu\in\vertices(\K)~|~\left<\mu,v_0\right>>0\}|+|\{\mu\in\vertices(\K)~|~\left<\mu,v_0\right>=0\}|\]
and thus
\[|\{\mu\in\vertices(\K)~|~\left<\mu,v_0\right>>0\}|=\frac12\left(2^n-\frac{3-(-1)^n}{2}\sum_{i=0}^k\binom{k}{i}^2\right).\]
Now, using the Theorem \ref{thm:fund_dom_An}, we find
\begin{eqnarray*}
|\vertices(\mathcal{L})|&=&|\{\left<\mu,v_0\right>>0\}|+|\{\left<\mu,v_0\right>=0\}| =\frac12\left(2^n+\frac{3-(-1)^n}{2}\sum_{i=0}^k\binom{k}{i}^2\right)\\
&=&\frac12\left(2^n+\frac{3-(-1)^n}{2}\binom{2k}{k}\right),
\end{eqnarray*}
as stated.
\end{proof}

\begin{prop}\label{prop:boundings}
When $n \geq 4$, the bounding hyperplanes of $\mathcal{L}$ are precisely the $2n$ hyperplanes that bound $\K$, together with $\{x \in V~|~\left<\lambda,v_0\right>=0\}$.
\end{prop}
\begin{proof}
It is clear that the hyperplane of the statement bounds $\mathcal{L}$ and that the other facets of $\mathcal{L}$ are contained in bounding hyperplanes for $\K$. Thus, it suffices to show that any bounding hyperplane of $\K$ contains at least $n$ vertices of $\mathcal{L}$, for in this case, these vertices generate a hyperplane that intersects $\K$ (resp. $\mathcal{L}$) in one of its facets, proving the result. To do this, we let $i\in I$ and we have to find at least $n$ vertices of $\mathcal{L}$ inside the hyperplanes $H_i$ and $H_i^0$.

Recalling the notation from the proof of the Theorem \ref{thm:fund_dom_An}, the hyperplane $H_i$ contains the vertices $0$, $\frac12\varpi^\vee_j$ (for $1\le j\le k$ such that $j\ne i$), $\frac13(\varpi^\vee_j+\varpi^\vee_{j'})$ (for $1\le j\le k$ with $j\ne i$ and $k'\le j'\le n$ with $j'\ne i$). Enumerating these subsets, we find that they contain at least $1+(k-1)+k(k-1)=k^2$ points. If $n\ne5$, then $k^2\ge n$ and we are done. Otherwise, we either have $i=3$, in which case the number of the vertices exhibited above is at least $1+2+2^2=7>5$, or there are at least $4$ of the above points, together with the additional $\varpi^\vee_3/2$. In any case, we indeed find at least $n$ vertices of $\mathcal{L}$ inside $H_i$.

Now, for the hyperplane $H_i^0$, we distinguish two cases. If $i\le k$, then $H_i^0$ contains $\frac13(\varpi^\vee_i+\varpi^\vee_j)$ (for $j\ne i$) and the barycenter $\frac{1}{n+1}\sum_j\varpi^\vee_j$, at least $n$ points indeed. Otherwise, we find in $H_i^0$ the points $\frac13(\varpi^\vee_i+\varpi^\vee_j)$ (for $1\le j \le k$), as well as $\frac14(\varpi^\vee_i+\varpi^\vee_j+\varpi^\vee_{j'})$ (for $1\le j\ne j'\le k$) which, together with the barycenter, form at least $k+k(k-1)+1=k^2+1\ge 2k+1\ge n$ vertices, finishing the proof.
\end{proof}

\begin{rem}
The previous proposition fails when $n\le3$. 
In fact, in these cases $v_0 = \alpha_1 - \alpha_n$ so that $(x,v_0) \geq 0$ is equivalent to $(x,\alpha_1) \ge (x,\alpha_n)$.
Then for $n=3$ (resp. $n=2$), two hyperplanes bounding $\K$ can be removed, and we have
\begin{align*}
\mathcal{L}&=\{x \in V~|~\forall i=1,2,3,
~0\le (x,\alpha_i) \le1- (x,\alpha_0) ,~ (x,\alpha_1) \geq (x,\alpha_3)\} \\
&=\{x \in V~|~ (x,\alpha_2) , (x,\alpha_3) \ge0,~ (x,\alpha_0+\alpha_1) , (x,\alpha_0+\alpha_2) \le1,~ (x,v_0) \ge0\},
\\
\left(\text{resp. }\right.\mathcal{L}&=\{x \in V~|~\forall i=1,2,~0\le (x,\alpha_i) \le1- (x,\alpha_0) ,~  (x,\alpha_1) \geq (x,\alpha_2) \} \\
&=\left.\{x \in V~|~ (x,\alpha_2) \ge0,~ (x,\alpha_0+\alpha_1) \le1,~ (x,v_0) \ge0\}\right).
\end{align*}
These simplifications come from the fact that $n=2,3$ are the exact two cases where $\phi_0$ acts as a genuine orthogonal reflection on $V$
and the resulting orthogonal decomposition of any vector in $\mathcal{L}$ allows to see that the two additional hyperplanes are in fact superfluous.
\end{rem}

\begin{rem}
As detailed in the preceding section, the support of the balanced root $v_0$ has to be chosen inside $J$ in order that $\vertices(\mathcal{L}) \subseteq \vertices(\mathcal{K})$. For $A_n$, we chose the support of $v_0$ to be maximal in order that the number of vertices of $\mathcal{L}$ to be minimal:
the bigger the $J_0$, the more restrictive the condition $|L \cap J^+_0| \geq |L \cap J^-_0|$.  For instance, if $J_0 = \{1, n\}$ and $v_0=\alpha_1-\alpha_n$, then the condition amounts to: if $n \in L$ then $1 \in L$.  The corresponding $\mathcal{L}_0$ 
%then the polytope
%\[\mathcal{L}%=\left\{\left.\frac{1}{|L|+1}\sum_{j\in L}\varpi^\vee_j~\right|~L\subseteq I,~\,\left|L\cap\{1\}\right|\geq\left|L\cap\{n\}\right|\right\}
%=\left\{\left.\frac{1}{|L|+1}\sum_{j\in L}\varpi^\vee_j~\right|
%~(n\in L~\Rightarrow~1\in L)
%\right\}\]
then has a maximal number of vertices, namely
\[|\vertices(\mathcal{L}_0)|=2^n-2^{n-2}=3\cdot2^{n-2}.\]
\end{rem}

\subsection{Type $D_{n\ge4}$}
%We study here the case of $D_n$, with $n\ge4$. 
The highest root is $\alpha_0=\alpha_1+2\alpha_2+\cdots+2\alpha_{n-2}+\alpha_{n-1}+\alpha_n$ so that $J:=\{1,n-1,n\}$ and the Komrakov-Premet polytope has vertices
\begin{equation}
\vertices(\K) =
\left\{\frac{\varpi^\vee_i}{2}\right\}_{i\in I}\cup
\left\{\frac{\varpi^\vee_1+\varpi^\vee_{n-1}}{3},
\frac{\varpi^\vee_1+\varpi^\vee_n}{3},
\frac{\varpi^\vee_{n-1}+\varpi^\vee_n}{3},
\frac{\varpi^\vee_1+\varpi^\vee_{n-1}+\varpi^\vee_n}{4}\right\}
\end{equation}
%The case $n \geq 5$ will be treated first, and we then refine our study for $n=4$.
By Table \ref{table:thm1}, when $n \geq 5$, $\Aut(D)$ is generated by one isometric involution $\phi_0$, which acts as the transposition $(n,n-1)$ of the simple root indices. It completes the fundamental group $\Omega$  in $\Aut(\A)$ in such a way that
\[
\Aut(\A)\simeq I_2(4)
\]
The unique choice for the support of a minuscule balanced root $v_0$ w.r.t.\ $\phi_0$ is $J_0 = \{n, n-1\}$. Let $J^+_0 = \{n-1 \}$, $J^-_0 = \{ n \}$, so that
$$
v_0:= \alpha_{n-1} -\alpha_n
$$
Note that $\phi_0$ is the identity on the hyperplane generated by 
$\alpha_1,\dotsc,\alpha_{n-2},\alpha_{n-1} + \alpha_n$, which is precisely the hyperplane orthogonal to $v_0$. It follows that $\phi_0$ is the orthogonal reflection around $\{ v_0 = 0 \}$. 
%Note that when $n$ is odd, we have $\phi_0=-w_0$ as in type $A_n$.

By Proposition \ref{prop:u0-slice}, $\mathcal{L}$ is a fundamental polytope for $\Aut(D)=\left<\phi_0\right>$ acting on $\mathcal{K}$. Theorem \ref{theo:vertex} then implies the following.

\begin{prop}\label{thm:fund_dom_Dn}
For $n \geq 5$, $\mathcal{L}$ is a fundamental polytope for the action of $\Aut(\A)$ in $\A$ such that
\begin{eqnarray}
\label{D-vert}
\vertices(\mathcal{L}) &=& \vertices(\K) \cap \mathcal{L} 
= \vertices(\K)\setminus\left\{\frac{\varpi^\vee_n}{2},\frac{\varpi^\vee_1+\varpi^\vee_n}{3}\right\} \nonumber\\
&=&
\left\{\frac{\varpi^\vee_i}{2}\right\}_{i<n}\cup\left\{\frac{\varpi^\vee_1+\varpi^\vee_{n-1}}{3},\frac{\varpi^\vee_{n-1}+\varpi^\vee_n}{3},\frac{\varpi^\vee_1+\varpi^\vee_{n-1}+\varpi^\vee_n}{4}\right\} \nonumber
\end{eqnarray}
In particular, $|\vertices(\mathcal{L})|=|\vertices(\K)|-2$.
\end{prop}
\begin{proof}
The condition $|L \cap J^+_0| \geq |L \cap J^-_0|$ amounts to: if $n \in L$ then $n-1 \in L$.
\end{proof}

\begin{prop}\label{prop:bounding_Fp_Dn}
When $n \geq 5$, the bounding hyperplanes of $\mathcal{L}$ are precisely the $H_i$'s for $i\ne n-1$, the $H_j^0$'s with $j\in\{1,n-1\}$, together with $\{x \in V~|~(x,v_0)=0\}$.
\end{prop}
\begin{proof}
We proceed as in type $A_n$. For $i<n-1$, the intersection $\mathcal{L}\cap H_i$ contains the origin, together with the $n-2$ points $\varpi^\vee_{k}/2$ for $k\notin\{i,n\}$ and $(\varpi^\vee_{n-1}+\varpi^\vee_n)/3$ and $\mathcal{L}\cap H_{n}$ contains all the points $\varpi^\vee_k/2$ for $k\ne n$, along with the origin. Next, if $j=1$, then the $n$ points $\varpi^\vee_1/2$, $(\varpi^\vee_1+\varpi^\vee_{n-1})/3$, $(\varpi^\vee_1+\varpi^\vee_{n-1}+\varpi^\vee_n)/4$ and $\varpi^\vee_k/2$ ($1<k<n-1$) belong to $\mathcal{L}\cap H_1^0$ and similarly for $j=n-1$, replacing the vertex $\varpi^\vee_1/2$ by $\varpi^\vee_{n-1}/2$. Finally, the $n$ points $0$, $\varpi^\vee_k/2$ ($k<n-1$) and $(\varpi^\vee_{n-1}+\varpi^\vee_n)/3$ annihilate $v_0$.

On the other hand, the intersection $\mathcal{L}\cap H_{n-1}$ (resp. $\mathcal{L}\cap H_n^0$) only contains the vertices $0$ and $\varpi^\vee_k/2$ for $k<n-1$ (resp. $(\varpi^\vee_{n-1}+\varpi^\vee_{n})/3$, $(\varpi^\vee_1+\varpi^\vee_{n-1}+\varpi^\vee_{n})/4$ and $\varpi^\vee_k/2$ ($1<k<n-1$), in number $n-1$.
\end{proof}

We now focus on the triality case $n=4$.
By Table \ref{table:thm1}, $\Aut(D)$ is generated by $\phi_0$, which acts as $(4,3)$ and $\phi_1$, which acts as $(3,1)$ on the simple root indices.
It completes the fundamental group $\Omega$ in $\Aut(\A)$ 
in such a way that
\[
\Aut(\A) \simeq \Sym_4
\]
the only case where $\Aut(\A)$ is not of dihedral type. 
The balanced minuscle root w.r.t.\ $\phi_0$ is given by
$$
v_0:=\alpha_3-\alpha_4
$$
and we have a essentially unique choice for a balanced minuscle root w.r.t.\ $\phi_1$, given by
$$
v_1:=\alpha_1-\alpha_3
$$

%We mimic the idea of Theorem \ref{thm:fund_dom_Dn}, adding one more ``root vector''. Define $\phi_1':=\phi_1^{\phi_0}=(1,3)$.
\begin{theo}\label{thm:fund_dom_D4}
The subset
\[
\mathcal{L}':=\{x\in \K~|~(v_0,x)\ge0,~(v_1,x)\ge0\}
\]
is a fundamental polytope for the action of $\Aut(\A)$ on $\A$, such that
\begin{align*}
\vertices(\mathcal{L}')&=
\vertices(\K) \cap \mathcal{L}'
= \vertices(\K)\setminus\left\{
\frac{\varpi^\vee_3}{2}, \frac{\varpi^\vee_4}{2},
\frac{\varpi^\vee_1+\varpi^\vee_4}{3}, \frac{\varpi^\vee_3+\varpi^\vee_4}{3}
\right\} 
\nonumber\\
&=
\left\{\frac{\varpi^\vee_1}{2},\frac{\varpi^\vee_2}{2},\frac{\varpi^\vee_1+\varpi^\vee_3}{3},\frac{\varpi^\vee_1+\varpi^\vee_3+\varpi^\vee_4}{4}\right\}
\end{align*}
In particular, $|\vertices(\mathcal{L}')|=|\vertices(\K)|-4$.
\end{theo}
\begin{proof}
Note that $\mathcal{L}' = \{x\in \mathcal{L}~|~(v_1,x)\ge0\}$.
The equalities $\phi_0(v_0)=-v_0$ and $\phi_1'(v_1)=-v_1$ ensure that $\mathcal{L}'$ is indeed a fundamental polytope for $\left<\phi_0,\phi_1\right> = \Aut(D)$. It only remains to prove that $\vertices(\mathcal{L}')\subseteq\vertices(\K)$.
By Theorem \ref{thm:fund_dom_Dn}, we already know that any vertex of $\mathcal{L}$ is a vertex of $\K$, so it suffices to see that $\vertices(\mathcal{L}')\subseteq\vertices(\mathcal{L})$. 
We let
\[
H_0^0:=\{x \in V~|~(x,v_0)=0\} \qquad 
H^1_0:=\{x \in V~|~(x,v_1)=0\}.
\]
If $\mu\in\vertices(\mathcal{L}')$, then there are $K\subseteq\{1,2,3,4\}$, $L\subseteq\{0,1,3,4\}$ with $|K|+|L|=3$ and
\[\{\mu\}=\mathcal{L}'\cap\left(H_0^1\cap\bigcap_{i\in K}H_i\cap\bigcap_{j\in L}H_j^0\right).\]
We indeed may assume that $\mu\in H^1_0$, because otherwise we'd already have $\mu\in\vertices(\mathcal{L})$. 
We may then write $\mu$ as
\[
\mu=\mu_1(\varpi^\vee_1+\varpi^\vee_3)+\mu_2\varpi^\vee_2+\mu_4\varpi^\vee_4,
\]
with $\mu_4=0$ as soon as $\mu_1=0$. We reason by filtering the possible size of $K$.
\begin{itemize}[label=$\bullet$]
\item If $K=\emptyset$, then either $\mu_1=\mu_4$ or $L=\{1,3,4\}$. In the latter case, we have $3\mu_1+2\mu_2+\mu_4=2\mu_1+2\mu_2+2\mu_4=1$, so $\mu_1=\mu_4$ in any case and thus $(\mu,v_0)=0$ and then $\mu\in\vertices(\mathcal{L})$.
\item If $K=\{1\}$ (resp. $K=\{3\}$), then $\mu\in H_3$ (resp. $\mu\in H_1$), so $\mu\in\vertices(\mathcal{L})$.

When $K=\{4\}$, we may assume that either $1\in L$ or $3\in L$ (otherwise, $4\in L\cap K$, so $\left<\mu,\alpha_0\right>=1$ and we find an additional hyperplane $H_j^0$ containing $\mu$). In both cases, we find $\mu_1=\mu_3=1$ and thus $L=\{1,3\}$. Since $\mu$ lies on an edge of $\mathcal{L}$, this means that it belongs in fact to a segment joining two points of $\mathcal{L}'$, because one of them must feature $\varpi^\vee_2$ and the only such vertex of $\mathcal{L}$ is $\varpi^\vee_2/2\in \mathcal{L}'$ and then, the other vertex has no choice but to be $(\varpi^\vee_1+\varpi^\vee_3)/3$. Because $\mu$ is extremal in $\mathcal{L}'$, it must be one of the endpoints of this segment, hence a vertex of $\mathcal{L}$.

If $K=\{2\}$, then either $0\in L$ i.e. $\mu_1=\mu_4$, so $4\mu_1=1$ and we're done, or $L$ contains at least $1$ or $3$ and thus $L=\{1,3\}$ again. An edge of $\mathcal{L}$ containing $\mu$ must have endpoints belonging to the set 
\[\left\{\frac{\varpi^\vee_1+\varpi^\vee_3}{3},\frac{\varpi^\vee_3+\varpi^\vee_4}{3},\frac{\varpi^\vee_1+\varpi^\vee_3+\varpi^\vee_4}{4}\right\}.\]
Since the first and third point belong to $\mathcal{L}'$, if these are the endpoints, we may conclude as above. So it remains to see what happens if the second point is an endpoint of the aforementioned segment. If the other one is $(\varpi^\vee_1+\varpi^\vee_3)/3$ (resp. $(\varpi^\vee_1+\varpi^\vee_3+\varpi^\vee_4)/4$), then $\mu_3=1/3=\mu_1$ (resp. $\mu_1=1/4=\mu_3$) and so $\mu$ itself is an endpoint, different from $(\varpi^\vee_3+\varpi^\vee_4)/3$, so $\mu\in\vertices(\mathcal{L})$.
\item Now, if $|K|=2$ then either $\mu_1=0$, in which case $\mu$ is on an edge of $\mathcal{L}$ whose at least one endpoint is $\varpi^\vee_2/2$, the other one must only feature $\varpi^\vee_4$, so $\mu_4=0$ and $\mu=\varpi^\vee_2/2\in\vertices(\mathcal{L})$. Otherwise, we have $K=\{2,4\}$, so we may assume that $1\in L$ or $3\in L$, so $\mu_1=1/3$ and again, $\mu\in\vertices(\mathcal{L})$.
\item Finally, the equality $|K|=3$ forces $\mu=0\in\vertices(\mathcal{L})$.
\end{itemize}
Finally, the conditions $|L \cap J^+_0| \geq |L \cap J^-_0|$ for the support of $v_0$ and $v_1$ amounts to: if $4 \in L$ then $3 \in L$ and if $3 \in L$ then $1 \in L$.
\end{proof}

\begin{prop}\label{prop:bounding_Fp_D4}
  The bounding hyperplanes of the fundamental polytope $\mathcal{L}'$ of
  {\rm Theorem~\ref{thm:fund_dom_D4}} are $H_2,H_4,H_1^0$ together with $\{x \in V~|~(x,v_0)=0\}$ and $\{x \in V~|~(x,v_1)=0\}$.
\end{prop}
\begin{proof}
The intersection $\mathcal{L}'\cap H_2$ (resp. $\mathcal{L}'\cap H_4$, $\mathcal{L}'\cap H_1^0$) contains $\vertices(\mathcal{L}')\setminus\{\varpi^\vee_2/2\}$ (resp. $\vertices(\mathcal{L}')\setminus\{(\varpi^\vee_1+\varpi^\vee_3+\varpi^\vee_4)/4\}$, $\vertices(\mathcal{L}')\setminus\{0\}$). The only vertices that do not belong to the remaining two intersection are respectively $(\varpi^\vee_1+\varpi^\vee_3)/3$ and $\varpi^\vee_1/2$.

All other hyperplanes bounding $\K$ contain at most three of the vertices of $\mathcal{L}'$, concluding the proof.
\end{proof}

\subsection{Type $E_6$}

Here, the highest root is $\alpha_0=\alpha_1+2\alpha_2+2\alpha_3+3\alpha_4+2\alpha_5+\alpha_6$, so that $J=\{1,6\}$ and the Komrakov--Premet polytope has vertices
\begin{align*}
\vertices(\K) 
&=
\left\{\frac{\varpi^\vee_i}{2}\right\}_{i\ne4}\cup\left\{\frac{\varpi^\vee_4}{3},\frac{\varpi^\vee_1+\varpi^\vee_6}{3}\right\}
\end{align*}
By Table \ref{table:thm1}, $\Aut(D)$ is generated by one isometric involution $\phi_0$, which acts as $(5,3)(6,1)$ of the simple root indices.
It completes the fundamental group $\Omega=\Z/3\Z$ in $\Aut(\A)$ in such a way that
\[
\Aut(\A)\simeq\Sym_3=I_2(3)
\]
The unique choice for the support of a minuscle balanced root $v_0$ w.r.t.\ $\phi_0$ is $J_0 = \{1, 6\}$. Let $J^+_0 = \{1 \}$, $J^-_0 = \{ 6 \}$ so that
$$
v_0:= \alpha_{1} -\alpha_6
$$

By Proposition \ref{prop:u0-slice}, $\mathcal{L}$ is a fundamental polytope for $\Aut(D)=\left<\phi_0\right>$ acting on $\mathcal{K}$. Theorem \ref{theo:vertex} then implies the following.

\begin{prop}
$\mathcal{L}$ is a fundamental polytope for the action of $\Aut(\A)$ in $\A$ such that
\[ \vertices(\mathcal{L}) = \vertices(\mathcal{K}) \cap \mathcal{L} 
  = \vertices(\K)\setminus\left\{\frac{\varpi^\vee_6}{2}\right\}.\]
\end{prop}
\begin{proof}
The condition $|L \cap J^+_0| \geq |L \cap J^-_0|$ amounts to: if $6 \in L$ then $1 \in L$.
\end{proof}

\begin{rem}\label{rem:minuscule_needed_E6}
For $v_0$ such that $\phi_0(v_0) = -v_0$ we could also have taken $v_0=\alpha_3-\alpha_5$, or even $v_0=\alpha_1+\alpha_3-\alpha_5-\alpha_6$, but in these cases the inclusion $\vertices(\mathcal{L})\subseteq \vertices(\K)$ no longer holds.
Indeed, if for instance $v_0=\alpha_3-\alpha_5$, then the point
\[\frac14(\varpi^\vee_3+\varpi^\vee_5)\in \bigcap_{i=1,2,4,6}H_i\cap \bigcap_{j=0,1,6}H_j^0\]
is the only vertex of $\mathcal{L}$ which is not a vertex of $\mathcal{K}$. If $v_0=\alpha_1+\alpha_3-\alpha_5-\alpha_6$, there are three of them: $(\varpi^\vee_3+\varpi^\vee_5)/4$, $(\varpi^\vee_3+\varpi^\vee_6)/4$ and $(\varpi^\vee_1+\varpi^\vee_5)/4$.
\end{rem}

\begin{prop}\label{prop:bounding_Fp_E6}
The bounding hyperplanes for $\mathcal{L}$ are precisely $H_0^0$ together with those of $\K$, except $H_1$ and $H_6^0$.
\end{prop}
\begin{proof}
For $i>1$, the only vertex of $\mathcal{L}$ not belonging to $H_i$ is the one involving only $\varpi^\vee_i$, and there are indeed six of them. Similarly, the hyperplane $H_1^0$ contains all vertices of $\mathcal{L}$ except $0$.
Now, the intersection $\mathcal{L}\cap H_1$ (resp. $\mathcal{L}\cap H_6^0$) does not contain the vertices $\varpi^\vee_1/2$ and $(\varpi^\vee_1+\varpi^\vee_6)/3$ (resp. $0$ and $\varpi^\vee_1/2$).
Finally, the intersection $\mathcal{L}\cap H_0^0$ contains all the vertices of $\mathcal{L}$ except $\varpi^\vee_1/2$.
\end{proof}

Collecting the results of this section we get Theorem B of the introduction.

\appendix
%%%%%%%%%%%%%%%%%
\section{Stratified centralizers}

Here we discuss some related negative results on stratified centralizers and equivariant triangulations.  

Given a subgroup $G \subseteq \E(V)$ denote the centralizer of $x \in V$ by
$$
G_x = \{ g \in G \mid gx = x \}
$$
In many instances, one wishes to know how $G_x$ varies when $x$ varies within a fundamental domain of $G$  (see Remark \ref{rem:centraliz}).  
Suppose that $G$ has a fundamental domain given by a polytope $\mathcal{F}$, which is stratified by its $k$-dimensional faces, $k=0,1,\ldots,n$.
The {\it regular face} of $x \in \mathcal{F}$ is the intersection of all the faces of $\mathcal{F}$ containing $x$. Note that the regular face of a
vertex of $\mathcal{F}$ is the vertex itself.

\begin{definition}\label{def:stratified_centralizers}
We say that $G$ has {\it stratified centralizers} w.r.t.\ the fundamental polytope $\mathcal{F}$ if $G_x = G_y$ whenever $x, y \in \mathcal{F}$ have the same regular face.
\end{definition}

This is well known to hold for the affine Weyl group $W_{\mathrm{aff}}$ w.r.t.\ $\A$. Indeed, $(W_{\mathrm{aff}})_x$ is generated by the reflections in the top dimensional faces of $\A$ containing the regular face of $x \in \A$ (see Proposition V.\S 3.1(vii) p.80 of \cite{Bo68}). 

What about the extended Weyl group $W_{\mathrm{ext}}$? For the Komrakov--Premet domain $\mathcal{K}$ we will see, by presenting counterexamples, that this is not always the case. We start with some preparatory remarks. Recall from Section 1 that
\begin{equation}\label{eq:wext-w-component} 
W_{\mathrm{aff}} = Q^\vee \rtimes W \qquad 
W_{\mathrm{ext}} := P^\vee  \rtimes W = \Omega \rtimes W_{\mathrm{aff}} 
\end{equation}

\begin{prop}
\label{prop:centraliz-ext}
For $x \in \A$, we have
\begin{equation}
(W_{\mathrm{ext}})_x = \Omega_x \rtimes (W_{\mathrm{aff}})_x.
\end{equation}
\end{prop}
\begin{proof}
By \eqref{eq:wext-w-component}, $\zeta \in W_{\mathrm{ext}}$ decomposes as $\zeta = \omega \eta$, for $\omega \in \Omega$ and $\eta \in W_{\mathrm{aff}}$. Since $x \in \A$, 
$\zeta x = x$ implies that $\eta x = \omega^{-1} x \in \A$ so that, $\eta x = x$ (see Proposition V.\S 3.2(I) p.80 of \cite{Bo68}). It follows that $\omega x = x$, which proves our claim.
\end{proof}

Since $W_{\mathrm{aff}}$ has stratified centralizers w.r.t.\ $\A$ and since the regular face of $x \in \K$ is contained in the regular face of $x \in \A$ (which is not necessarily of the same dimension), it follows that $W_{\mathrm{aff}}$ has stratified centralizers w.r.t.\ $\K$. Proposition \ref{prop:centraliz-ext} then implies the following.
%\orange{I don't understand this remark: in type $A_3$ for instance, the faces of $\mathcal{K}$ which are supported in $\mathcal{F}_{1,2,3}$ (with the notation of the Example \ref{ex:A3_non_stratified}) are not contained in regular faces of $\mathcal{A}$ (or at least, not of the same dimension). Or did I miss something?}

\begin{cor}
\label{cor:wext}
The extended Weyl group $W_{\mathrm{ext}}$ has stratified centralizers w.r.t.\ $\K$ if and only if its fundamental group $\Omega$ does.
\end{cor}

Consider the Kac coordinates of $\A$ given by 
%\red{KH: This looks like the notation for simplices, bad.}
$$
[b_0, b_1, \ldots, b_n] = b_1 \varpi^\vee_1 + \ldots + b_n \varpi^\vee_n
\quad\text{with}\quad b_0, b_1, \ldots, b_n \geq 0, \quad 
b_0 + m_1 b_1 + \cdots + m_n b_m = 1 
$$
The top dimensional faces of $\A$ are given by
$$
\left(\alpha_i, \, [b_0, b_1, \ldots, b_n] \right) = b_i = 0 \quad (i=1,\ldots,n),
\qquad
\left(\alpha_0, \, [b_0, b_1, \ldots, b_n] \right) = 
m_1 b_1 + \cdots + m_n b_m = 1
$$
thus, have Kac coordinates $b_i = 0$  $(i=0,\ldots,n)$.
Furthermore, $\K$ has top dimensional faces $H^0_j  \cap \K$ given by
\begin{eqnarray*}
\left(\alpha_0 + \alpha_j,\, [b_0, b_1, \ldots, b_n] \right) = m_1 b_1 + \cdots + m_n b_m + b_j = 1 \quad (j \in J)
\end{eqnarray*}
thus, with Kac coordinates $b_0 = b_j$ $(j \in J)$.

\begin{theo}
For type $A_n$, the extended Weyl group $W_{\mathrm{ext}}$ has stratified centralizers w.r.t.\ the Komrakov--Premet domain $\K$ if and only if $n+1$ is prime.
\end{theo}

\begin{proof}
  For type $A_n$ we have $J=I$ and recall from Table \ref{table:omega} that $\Omega \simeq \Z_{n+1}$ generated by the cyclic permutation
    \[ \omega_1 = (n,n-1,\ldots,1,0),\]
    of simple roots corresponding to the extended Dynkin diagram, which acts in the Kac coordinates of $\A$ as the shift
$$
\omega_1 [b_0, b_1, \ldots, b_n] = [b_1, b_2, \ldots, b_n, b_0]
$$
For $0 \leq k \leq n+1$, it follows that $\omega_1^k$ is given by
$$
\omega_1^k [b_0, b_1, \ldots, b_n] = [b_k, b_{k+1},
\ldots, b_n, b_0, \ldots, b_{k-2}, b_{k-1}]
$$
with fixed point set
$$
\mathrm{Fix}(\omega_1^k) : = 
\A \cap
\Ker( \omega_1^k - 1 ) = \left\{ [b_0, b_1, \ldots, b_n] \mid b_i = b_{i + k \!\!\mod n+1} \right\}
$$
satisfying at least the $\ell := (n+1) - k$ independent equations
$$
b_0 = b_k, \quad b_1 = b_{k+1}, \quad \ldots \quad b_{\ell - 1} = b_n
$$
where $\ell + k = n + 1$, so that $\codim \mathrm{Fix}(\omega_1^k) \geq (n + 1) - k$.

%\red{KH: I don't understand this codimension argument; please check.
%  For $n = 2$ we obtain the same fixed points for $k = 1$ and $k = 2$.}
Also, $[b_0, b_1, \ldots, b_n] \in \mathrm{Fix}(\omega_1^k)$ implies $b_0 = b_k = b_{2k} = \cdots$ so that, by the comments preceding the theorem, $\mathrm{Fix}(\omega_1^k)\cap\K$ is contained in the faces of $\K$ supported by 
$H^0_{k}, H^0_{2k}, \ldots$, thus in the face 
$$
\mathcal{F}_k = \K \cap \bigcap_{j=1}^n  H^0_{j k \!\!\mod n+1}
$$
whose codimension is (order of $k$ in $\Z_{n+1}$) $-1$. 

For $n+1$ prime, we claim that $\Omega$ has stratified center.
Indeed, in this case, the only subgroups of $\Omega$ are the trivial ones, thus $\Omega_x$ is either $\Omega$ or $1$.  Note that $\Omega_x = \Omega$ precisely when $x$ is a fixed point of the generator $\omega_1$ whose only fixed point, by the previous paragraph, is the baricenter of $\A$, given by
$\{ x_* \} = \bigcap_{j=1}^n H^0_j$,
so that $x_* = [b_0, b_0, \ldots, b_0]$ with $b_0 = 1/(n+1)$. 
Thus if $x = x_*$ then $\Omega_x = \Omega$, otherwise $\Omega_x = 1$. By the above intersection, $x_*$ is a vertex of $\K$, thus a regular face of $\K$, and our claim follows.

For $n+1$ composite, we claim that $\Omega$ does not have stratified center.  Indeed, let $n+1 = k p$, with $k \geq p \geq 2$ coprime, so that the order
of $k$ in $\Z_{n+1}$ is $p$.
By the first paragraph, $\mathrm{Fix}(\omega_1^k) \cap \K$ is contained in the face $\mathcal{F}_k$ of $\K$ with codimension
%\red{KH: please clarify; see above}
\begin{eqnarray*}
\codim_{\mathcal{F}_k} \mathrm{Fix}(\omega_1^k) \cap \K 
& \geq & ((n+1)-k) - (p-1) \\
&=& kp - k - p + 1 \\
&=&  k(p-1) - p + 1 \geq k - p + 1 \geq 1
\end{eqnarray*}
It follows that $\mathcal{F}_k$ is the regular face of points $x \in 
\mathcal{F}_k \cap \mathrm{Fix}(\omega_1^k)$ with $\left< \omega_1^k \right > \leq \Omega_x$ and of $y \in \mathcal{F}_k \setminus \mathrm{Fix}(\omega_1^k)$ with $\left< \omega_1^k \right > \not\leq \Omega_{y}$,  so that our claim follows from Corollary \ref{cor:wext}.

\end{proof}

\begin{example}\label{ex:A3_non_stratified}
Let us picture how the stratified centralizers property for $\Omega$ fails for the $A_3$ type in the previous proof.

Denote by $\widecheck{ab}$ the baricenter of the segment
$\overline{ab}$.
%\red{KH: I don't like this notation for segments;
%$\overline{ab}$ is standard} using this notation 
The top faces of $\K$ are supported on the triangles $\mathcal{F}_i = H^0_i \cap \A$ given by
$$
\mathcal{F}_1\text{ with vertices }\widecheck{01},2,3, \quad
\mathcal{F}_2\text{ with vertices }\widecheck{02},1,3, \quad
\mathcal{F}_3\text{ with vertices }\widecheck{03},1,2.
$$
\begin{center}
\def\svgwidth{3.5cm}
%\def\svgwidth{19cm}
%% Creator: Inkscape 1.0.1 (c497b03c, 2020-09-10), www.inkscape.org
%% Accompanies image file 'tetrahedron.pdf' (pdf, eps, ps)
\begingroup%
  \makeatletter%
  \providecommand\color[2][]{%
    \errmessage{(Inkscape) Color is used for the text in Inkscape, but the package 'color.sty' is not loaded}%
    \renewcommand\color[2][]{}%
  }%
  \providecommand\transparent[1]{%
    \errmessage{(Inkscape) Transparency is used (non-zero) for the text in Inkscape, but the package 'transparent.sty' is not loaded}%
    \renewcommand\transparent[1]{}%
  }%
  \providecommand\rotatebox[2]{#2}%
  \newcommand*\fsize{\dimexpr\f@size pt\relax}%
  \newcommand*\lineheight[1]{\fontsize{\fsize}{#1\fsize}\selectfont}%
  \ifx\svgwidth\undefined%
    \setlength{\unitlength}{205.13668874bp}%
    \ifx\svgscale\undefined%
      \relax%
    \else%
      \setlength{\unitlength}{\unitlength * \real{\svgscale}}%
    \fi%
  \else%
    \setlength{\unitlength}{\svgwidth}%
  \fi%
  \global\let\svgwidth\undefined%
  \global\let\svgscale\undefined%
  \makeatother%
  \begin{picture}(1,1.0096533)%
    \lineheight{1}%
    \setlength\tabcolsep{0pt}%
    \put(0,0){\includegraphics[width=\unitlength,page=1]{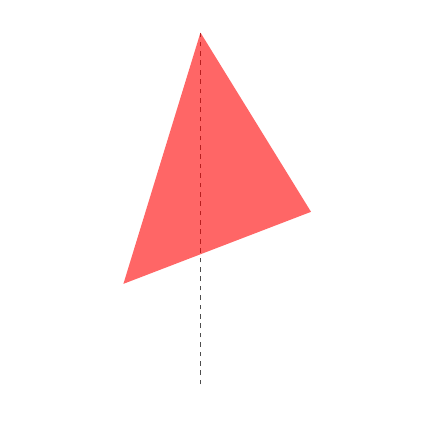}}%
    \put(0.44395926,0.00602685){\makebox(0,0)[lt]{\lineheight{1.25}\smash{\begin{tabular}[t]{l}$0$\end{tabular}}}}%
    \put(0.75107111,0.4813192){\makebox(0,0)[lt]{\lineheight{1.25}\smash{\begin{tabular}[t]{l}$1$\end{tabular}}}}%
    \put(-0.00208512,0.59100174){\makebox(0,0)[lt]{\lineheight{1.25}\smash{\begin{tabular}[t]{l}$2$\end{tabular}}}}%
    \put(0.44395894,0.96392354){\makebox(0,0)[lt]{\lineheight{1.25}\smash{\begin{tabular}[t]{l}$3$\end{tabular}}}}%
    \put(0.12953478,0.2912023){\makebox(0,0)[lt]{\lineheight{1.25}\smash{\begin{tabular}[t]{l}$\widecheck{02}$\end{tabular}}}}%
    \put(0.61213965,0.73724625){\makebox(0,0)[lt]{\lineheight{1.25}\smash{\begin{tabular}[t]{l}$\widecheck{13}$\end{tabular}}}}%
    \put(0,0){\includegraphics[width=\unitlength,page=2]{tetrahedron.pdf}}%
  \end{picture}%
\endgroup%

%The situation for the root system $B_2$.
\end{center}
We have that $\Omega \simeq \Z_4$ is generated by $\omega_1 = (3,2,1,0)$.
Let $\sigma := \omega_1^2 = (02)(13)$ is a product of commuting reflections: $(02)$ around $\mathcal{F}_2$, $(13)$ around the triangle
$$
\mathcal{F}_0\text{ with vertices }\widecheck{13},0,2.
$$
which is not a face of $\K$.

It follows that
$$
\mathrm{Fix}(\sigma) = \mathrm{Fix}(02) \cap \mathrm{Fix}(13) 
= \mathcal{F}_2 \cap \mathcal{F}_0
$$
is the segment from $\widecheck{02}$ to $\widecheck{13}$.
It also follows that $\mathrm{Fix}(\sigma) \cap \K \subseteq \mathcal{F}_2 \cap \K$ with codimension~1.  
Thus $\mathcal{F}_2' :=\mathcal{F}_2 \cap \mathcal{K}$ is the regular face of points on the segment $x \in \mathcal{F}_2' \cap \mathrm{Fix}(\sigma)$ with $\left< \sigma \right > \leq \Omega_x$ and of points outside the segment $y \in \mathcal{F}_2' \setminus \mathrm{Fix}(\sigma)$ with $\left< \sigma \right > \not\leq \Omega_{y}$. It follows that $\Omega$ does not have stratified centralizers w.r.t.\ $\mathcal{K}$.

%
%Taking $x$ on the interior of $\mathrm{Fix}(\sigma) \cap \K$ w.r.t.\ the face $\mathcal{F}_2 \cap \K$ we have $\Omega_x = \left< \sigma \right>$, while taking $x'$ on the complement $\mathcal{F}_2 \cap \K \setminus \mathrm{Fix}(\sigma)$ we have $\Omega_{x'} = 1$. \orange{Maybe add a complete figure with annotations? Things are a bit hard to vizualise.}
\end{example}

In the previous example, note that the failure in the stratification occurs because of the lack of the face $\mathcal{F}_0$ in $\K$, fixed by $\phi_0 = (13)$. But since $\phi_0$ is the generator of $\Aut(D)$, $\mathcal{F}_0$ is precisely the additional face of the 
fundamental domain $\mathcal{L}$ for $\Aut(\A)$ obtained from the maximal balanced minuscule root $v_0 = \alpha_1 - \alpha_3$.  Thus, $\mathrm{Fix}(\sigma)$ is the intersection of faces of the fundamental domain $\mathcal{L}$ of $\Aut(\A)$.

This raises the question if extending $\Omega$ to $\Aut(\A)$, one gets stratified centralizers w.r.t.\ $\mathcal{L}$?  
%\red{$\mathcal{L}_0$?} ?
The following example shows that is not the case.

\begin{example}
Consider the root system $A_3$ and the notation of the previous example. By Table \ref{table:thm1} we have that $\Aut(\A) \simeq S_3$ is generated by $\tau_0 := \phi_0$ and $\tau_1 := \omega_1 \phi_0$.  Then 
$$
\tau_1[b_0,b_1,b_2,b_3] = \omega_1 \phi_0[b_0,b_1,b_2,b_3]  = 
\omega_1 [b_0,b_3,b_2,b_1] = [b_3,b_2,b_1,b_0] = (03)(12)
$$
so that $\tau_1$ is the product of commuting reflections: $(03)$ around $\mathcal{F}_3$, $(12)$ around the triangle 
$$
\mathcal{F}_4\text{ with vertices }\widecheck{12},0,3.
$$
which is not a face of $\mathcal{L}$.
It follows that
$$
\mathrm{Fix}(\tau_1) = \mathrm{Fix}(03) \cap \mathrm{Fix}(12) 
= \mathcal{F}_3 \cap \mathcal{F}_4
$$
is the segment from $\widecheck{03}$ to $\widecheck{12}$. 

As in the previous example, it follows that $\mathrm{Fix}(\tau_1) \cap \K \subseteq \mathcal{F}_3 \cap \mathcal{L}$ with codimension 1.  Thus, $\mathcal{F}_3' :=\mathcal{F}_3 \cap \mathcal{L}$ is the regular face of points on the segment $x \in \mathcal{F}_3' \cap \mathrm{Fix}(\tau_1)$ with $\left< \tau_1 \right > \leq \Omega_x$ and of points outside the segment $y \in \mathcal{F}_3' \setminus \mathrm{Fix}(\tau_1)$ with $\left< \tau_1 \right > \not\leq \Omega_{y}$. It follows that $\Aut(\A)$ does not have stratified centralizers w.r.t.\ $\mathcal{L}$.
\end{example}

%\begin{example*}
%$\Aut(\A)$ for $A_2$ e $G_2$.
%\end{example*}

\begin{rem}
\label{rem:centraliz}
In \cite{SS18, SP24} 
the authors consider the centralizer of $x$ in the Weyl group modulo, respectively, the coroot and coweight lattice, given by
\begin{eqnarray*}
W^x_\circ := \{ w \in W \mid wx = x + \gamma, \quad \gamma \in Q^\vee \} 
&\vartriangleleft&
W^x := \{ w \in W \mid wx = x + \gamma, \quad \gamma \in P^\vee \}
\end{eqnarray*}

They relate to centralizers on the affine and the extended Weyl group as follows.

\begin{prop}
Projection onto the $W$-component \eqref{eq:wext-w-component}
restricts to isomorphisms
\begin{equation}
\label{eq:proj-centralizer}
(W_{\mathrm{aff}})_x \to W^x_\circ
\qquad
(W_{\mathrm{ext}})_x \to W^x
\end{equation}
from which we get the isomorphisms
\begin{equation}
\Omega_x  \leftarrow (W_{\mathrm{ext}})_x /  (W_{\mathrm{aff}})_x
\to W^x/W^x_\circ
\end{equation}
\end{prop}
\begin{proof}

  Let $t_\gamma w \in W_{\mathrm{aff}}$, with $w \in W$, $\gamma \in Q^\vee$, 
such that $t_\gamma w x = wx + \gamma = x$. Then $wx = x - \gamma$, so that $w \in W^x_\circ$ and the projection restricts to $(W_{\mathrm{aff}})_x \to W^x_\circ$.
%\red{KH: don't understand this conclusion}
It is clearly injective, since $w=1$ implies that $x + \gamma = x$, so that $\gamma = 0$. And it is clearly surjective, since $w$ such that $wx = x + \gamma$ is the $W$-component projection of $t_{-\gamma} w \in (W_{\mathrm{aff}})_x$.  For the second isomorphism, the argument is the same, taking $\gamma \in P^\vee$.
The first line of isomorphisms together with Proposition \ref{prop:centraliz-ext} then implies the second line of isomorphisms.
\end{proof}

It follows from the previous considerations in this section that
$W^x_\circ$ has stratified centralizers w.r.t.\ $\K$ (and $\A$) but $W^x$ in general does not.
\end{rem}

We finish by mentioning some implications of the notion of stratified centralizers in equivariant algebraic topology. Recall from \cite{GaPhD} that, if $G$ is a group and $X$ is a $G$-space, then the (singular) cohomology $H^*(X;\Z)$ naturally inherits the structure of a $\Z[G]$-module. Under some mild hypotheses,
namely $G$ discrete and $X$ locally contractible,
%\footnote{Namely, $G$ should be discrete and $X$ should be locally contractible.},
this module is the cohomology of the derived global section complex $R\Gamma(X;\Z)$, an element of the derived category $D^b(\Z[G])$ of $\Z[G]$-modules, so that the action of $G$ on $H^*(X;\Z)$ may be lifted to the level of complexes in the derived category. In contrast to the classical setting, we do not have an isomorphism $H^*(X;\Z)\simeq R\Gamma(X;\Z)$ in $D^b(\Z[G])$, so the action of $G$ on $R\Gamma(X;\Z)$ contains {\it a priori} a more precise information on the $G$-action on $X$ than the cohomology $H^*(X;\Z)$.

The problem then becomes to find a convenient combinatorial description of the (huge) complex $R\Gamma(X;\Z)$ and this is where the notion of $G$-equivariant CW-complex comes into play. A classical CW-structure on $X$ is called {\it $G$-equivariant} if $G$ acts on the set of cells of the structure and if the following local-to-global condition is satisfied:
\begin{equation}\label{eq:cond}
\text{For each $x\in X$, if $e\subseteq X$ is the (unique) cell containing $x$, then we have $G_x=G_e$.}
\end{equation}
In other words, if an element of $G$ 
stabilizes a cell,
%leaves a cell globally invariant, 
then it must fix any of its points. To a $G$-equivariant CW-structure is attached its cellular cochain complex $C^*_{\rm cell}(X;\Z)$, a complex of $\Z[G]$-modules. As detailed in \cite{GaPhD}, this complex is isomorphic to $R\Gamma(X;\Z)$ in $D^b(\Z[G])$, providing a nice combinatorial interpretation of the derived sections, which is moreover well-defined up to equivariant homotopy. This motivates the study and construction of $G$-equivariant CW-structures.

Now, when $G\subseteq\E(V)$ acts on a Euclidean space $V$, with fundamental polytope $\mathcal{F}$, then the ($G$-translates of the) face lattice of \cancel{the} $\mathcal{F}$ induces a $G$-equivariant simplicial structure on $V$ in the above sense, if and only if $G$ has stratified centralizers, in the sense of the Definition \ref{def:stratified_centralizers}.

The fact that the affine Weyl group $W_{\rm aff}$ has stratified centralizers has been exploited by the second author in \cite{Ga23} to obtain a triangulation of a maximal torus of a simply-connected compact Lie group, equivariant with respect to the action of the Weyl group. In the adjoint case, because the extended group $W_{\rm ext}$ does not satisfy this property w.r.t. the Komrakov--Premet polytope $\mathcal{K}$, the second author rather considered the {\it barycentric subdivision} of the alcove $\mathcal{A}$, \cite[Theorem 3.2.3]{Ga23}.

Thus, the fact that $\Aut(\mathcal{A})$ does not have stratified centralizers w.r.t. the fundamental sub-polytope $\mathcal{L}\subseteq\mathcal{K}$ introduced above (even while $\Aut(\mathcal{A})$ is abstract Coxeter), can be reformulated by saying that the face lattice of $\mathcal{L}$ does {\it not} induce an
$(\Aut(\mathcal{A})\ltimes W)$-equivariant triangulation of $V$. By \cite[Lemma 3.2.1]{Ga23}, the barycentric subdivision of $\mathcal{A}$ does such a job, but has a quite intricate combinatorics, in comparison to the polytope $\mathcal{L}$.

\end{document}